\documentclass[a4paper,11pt]{amsart}
\usepackage[
  top=2cm,
  bottom=2cm,
  left=2.5cm,
  right=2.5cm,
  headheight=14pt,
  headsep=1cm,
  footskip=1cm
]{geometry}
\usepackage{amsmath, epsfig, verbatim}
\usepackage{amsfonts}
\usepackage{amsthm}
\usepackage{mathtools}
\usepackage{amsmath}
\usepackage{amssymb} 
\usepackage{graphics,graphicx}
\usepackage{epstopdf}
\usepackage{listings}
\usepackage{xcolor}
\usepackage{tikz}
\usepackage{subcaption}
\usepackage{pgfplots}
\usepackage{caption}
\usepackage{booktabs}
\usepackage{tabularx}
\usepackage{hyperref}   
\usepackage{doi}
\usepackage{enumitem}
\usepackage[numbers,sort&compress]{natbib}
\usepackage[linesnumbered,ruled,vlined]{algorithm2e}
\usetikzlibrary{calc}
\allowdisplaybreaks
\definecolor{codegray}{rgb}{0.95,0.95,0.95}
\definecolor{pykeyword}{rgb}{0.13,0.13,1}
\definecolor{pystring}{rgb}{0.58,0,0.82}

\lstdefinestyle{pythonstyle}{
    backgroundcolor=\color{codegray},
    language=Python,
    basicstyle=\ttfamily\small,
    keywordstyle=\color{pykeyword}\bfseries,
    stringstyle=\color{pystring},
    commentstyle=\color{gray},
    showstringspaces=false,
    numbers=left,
    numberstyle=\tiny,
    frame=single,
    breaklines=true,
    tabsize=4,
}

\numberwithin{equation}{section}
\theoremstyle{plain}
\newtheorem{theorem}{Theorem}
\newtheorem{defn}[theorem]{Definition}
\newtheorem{coro}[theorem]{Corollary}
\newtheorem{lemma}[theorem]{Lemma}
\newtheorem{prop}[theorem]{Proposition}
\newtheorem{remark}[theorem]{Remark}

\begin{document}
\title{ $k$-Mutual Visibility in Graphs} 
\author{Tonny K B}
\address{Tonny K B, Department of Mathematics, College of Engineering Trivandrum, Thiruvananthapuram, Kerala, India, 695016.}
\email{tonnykbd@cet.ac.in}
\author{Shikhi M}
\address{Shikhi M, Department of Mathematics, College of Engineering Trivandrum, Thiruvananthapuram, Kerala, India, 695016.}
\email{shikhim@cet.ac.in}
\begin{abstract}
In this paper, we introduce the notion of $k$-mutual visibility, a relaxation of classical mutual visibility in which every pair of selected vertices is joined by a shortest path containing at most $k$ internal vertices of the selected set. This parameterized concept naturally generalizes classical mutual visibility and provides a graded notion of obstruction tolerance. We define the $k$-mutual visibility number $\mu_k(G)$ and establish its fundamental properties. We derive bounds on $\mu_k(G)$ in terms of graph parameters such as diameter and girth, and determine its exact value for several fundamental graph classes. We further investigate $k$-mutual visibility in convex subgraphs and characterize it in block graphs by introducing the notion of $k$-admissible sets in the associated block-cutpoint tree. We present a polynomial-time algorithm, \textsc{kMV}, that recognizes whether a given subset $S\subseteq V(G)$ is a $k$-mutual visibility set of $G$. We also formulate the \textsc{$k$-Mutual Visibility} decision problem and prove that it is NP-complete. Finally, we define the $k$-mutual visibility covering number $\tau_k(G)$ and establish several of its fundamental properties.
\end{abstract}
\subjclass[2020]{05C12, 05C69, 05C85, 68Q17}
\keywords{mutual-visibility set, $k$-mutual visible set, $k$-mutual visibility number, $k$-mutual visibility covering number}
\maketitle
\section{Introduction}
Let $G=(V,E)$ be a simple graph and let $X \subseteq V$. 
Two vertices $u,v \in V$ are said to be $X$-visible \cite{MV_3} if there exists a shortest $(u, v)$-path $P$ such that none of the internal vertices of $P$ belong to $X$; equivalently,
$V(P)\cap X \subseteq \{u,v\}$.
A subset $X \subseteq V$ is called a mutual-visibility set \cite{Stefano} of $G$ if every pair of distinct vertices in $X$ is $X$-visible. 
The maximum cardinality of a mutual-visibility set in $G$ is called the mutual-visibility number of $G$ and is denoted by $\mu(G)$. 
Moreover, a vertex $u \in V(G)\setminus X$ is said to be $X$-visible if it is $X$-visible with respect to every vertex $v \in X$. This graph-theoretic concept has attracted considerable attention in the literature ~\cite{MV_1, MV_2, MV_3, MV_4, MV_5, MV_6, MV_7, MV_8, MV_9, TMV_1, TMV_2, VP_1}. 
In addition, a number of extensions of mutual visibility, including dual mutual visibility, 
outer mutual visibility, and total mutual visibility, have been introduced and studied 
in ~\cite{MV_10}.

The notion of mutual visibility in graphs captures the idea that pairs of vertices can ``see'' each other along shortest paths that avoid all other vertices of the chosen set. 
While this concept is natural and well motivated, it is also rather restrictive, as the presence of a single additional vertex on every shortest path between two vertices completely destroys their visibility.
In many graph classes, this strict blocking condition significantly limits the size of mutual visibility sets.

To overcome this limitation, we introduce the concept of $k$-mutual visibility, which allows a controlled level of tolerance.
Instead of requiring shortest paths to avoid the entire set, we permit up to $k$ internal vertices of the set to lie on a shortest $(u,v)$-path.
Thus, the parameter $k$ measures the tolerance of the set with respect to internal obstructions along geodesics.
When $k=0$, the definition coincides with classical mutual visibility, while larger values of $k$ provide a gradual relaxation of this condition.

This tolerance-based framework offers several advantages.
First, it yields a hierarchy of parameters that interpolates between strict visibility and unrestricted shortest-path connectivity.
Second, it allows for larger and more stable visibility sets in graphs where geodesics are highly constrained.
Finally, the parameter $k$ provides additional flexibility for structural and extremal investigations, enabling finer distinctions between graph classes.
In particular, in $k$-mutual visibility the parameter $k$ allows a trade-off between the size of a mutually visible set and the strictness of visibility along shortest paths.

In this paper, we define the $k$-mutual visibility number $\mu_k(G)$ of a graph $G$ and investigate 
its fundamental properties, including monotonicity with respect to $k$ and stabilization for 
sufficiently large values of $k$. 
We establish general bounds on $\mu_k(G)$ in terms of standard graph parameters 
such as diameter, maximum degree, and girth. In particular, we revisit several results from Di Stefano \cite{Stefano} and reinterpret them 
within the framework of $k$-mutual visibility, obtaining corresponding extensions under the new 
definition. We analyze $(X,k)$-visibility in convex graphs and obtain exact values of $\mu_k(G)$ 
for some fundamental graph families, including paths, cycles, and complete bipartite graphs.
For block graphs, we exploit their structural representation via the block--cutpoint tree and 
introduce a suitable notion of $k$-admissibility that captures $k$-mutual visibility through 
this tree structure. 
This approach enables a characterization of $k$-mutual visibility in block graphs and facilitates 
the determination of $\mu_k(G)$ for this class. In addition, we present a polynomial-time algorithm, \textsc{kMV}, that determines whether a given subset $S\subseteq V(G)$ is a $k$-mutual visibility set of $G$ in time $O\bigl(|S|(|V(G)|+|E(G)|)+|S|^2\bigr)$. We also study the associated decision problem and prove that the \textsc{$k$-Mutual Visibility} problem is NP-complete.
 
\section{Notations and preliminaries}
Let $G=(V, E)$ be a finite, simple, undirected graph with vertex set $V(G)$ and edge set $E(G)$. 
Unless stated otherwise, all graphs considered in this paper are assumed to be connected, so that 
a path exists between every pair of vertices. The maximum degree of a graph $G$ is denoted by $\Delta(G)$.

For a vertex $u \in V(G)$, the \emph{open neighborhood} of $u$ is denoted by $N(u)$ and is defined as 
$N(u)=\{v \in V(G) : uv \in E(G)\}$. 
The \emph{closed neighborhood} of $u$ is denoted by $N[u]$ and is given by 
$N[u]=N(u)\cup\{u\}$.
The \emph{complete graph} on $n$ vertices, denoted by $K_n$, is the graph in which every pair of distinct 
vertices is joined by an edge, and hence contains $\binom{n}{2}$ edges. 
A sequence of vertices $(u_0,u_1,\ldots,u_n)$ is called a \emph{$(u_0,u_n)$-path} in $G$ if 
$u_i u_{i+1}\in E(G)$ for all $i=0,1,\ldots,n-1$. 
A \emph{cycle} in a graph $G$ is obtained by adding the edge $u_0u_n$ to a $(u_0,u_n)$-path. 
If $G$ itself is a path on $n$ vertices, it is denoted by $P_n$, while a cycle on $n$ vertices is 
denoted by $C_n$. The length of a cycle is equal to the number of edges (or, equivalently, vertices) 
in that cycle. The \emph{girth} of a graph $G$, denoted by $g$, is the length of a shortest cycle in $G$. A path $P=(x_0, x_1, \cdots, x_\ell)$ in a graph $G$ is called \emph{isometric} if 
$d_G(x_i,x_j)=|i-j|$ for all $0\le i<j\le \ell$.

A shortest path between two vertices of a graph is called a geodesic. 
The distance between vertices $u$ and $v$ in $G$, denoted by $d_G(u,v)$, is the length of a 
geodesic joining them. 
The diameter of a graph $G$, written as ${\rm diam}(G)$, is the maximum distance between any pair 
of vertices of $G$.
Let $G$ be a graph. A subgraph $H$ of $G$ is said to be \emph{convex} if, for every pair of vertices $u,v\in V(H)$, every shortest $(u,v)$-path in $G$ is contained in $H$. Moreover, if $H$ is convex, then $d_H(u,v)=d_G(u,v)$ for all $u,v\in V(H)$. For a vertex set $S\subseteq V(G)$, the \emph{convex hull} of $S$, denoted by
$\mathrm{hull}_G(S)$, is the smallest convex subgraph of $G$ containing $S$.
Equivalently, $\mathrm{hull}_G(S)$ is obtained by iteratively adding to $S$
all vertices that lie on a shortest path in $G$ between two vertices already in the set. When the ambient graph $G$ is clear from the context, we write $\mathrm{hull}(S)$.

A set $S\subseteq V(G)$ is said to be in \emph{general position} if no vertex of $S$
lies on a shortest path between two other vertices of $S$. Equivalently, $S$ is in general position if for every three distinct vertices $u,v,w\in S$, the vertex $v$ does not belong to any shortest $(u,w)$-path in $G$. The \emph{general position number} of $G$, denoted by $gp(G)$, is the maximum
cardinality of a general position set in $G$. 

A \emph{block} of $G$ is a maximal $2$-connected subgraph of $G$.
We denote by $\mathcal{B}$ the set of all blocks of $G$.
The \emph{block-cutpoint graph} of $G$ is the bipartite graph $H$ whose vertex set is partitioned
into two parts $A$ and $\mathcal{B}'$, where $A$ consists of the articulation vertices of $G$
and $\mathcal{B}'$ consists of a vertex $b$ for each block $B \in \mathcal{B}$.
For $b \in \mathcal{B}'$ an edge $vb$ is included in $H$ if and only if 
$v \in A \cap V(B)$, where $B$ is the block corresponding to the vertex $b$. It is well known that if $G$ is connected, then its block-cutpoint graph is a tree.  A graph $G$ is called a \emph{block graph} if every block of $G$ is a clique.

Let $G$ and $H$ be two graphs. The \emph{corona} of $G$ and $H$, denoted by $G \odot H$, is the graph obtained by taking one copy of $G$ and $|V(G)|$ copies of $H$, and for each vertex $v$ in $G$, joining $v$ to each vertex in the corresponding copy of $H$ associated with  $v$.
\section{ $k$-Mutual Visibility in a Graph}
\begin{defn}
Let $G$ be a graph and let $X \subseteq V(G)$.
Two vertices $u,v \in V(G)$ are said to be $(X,k)$-visible if there exists a shortest $(u, v)$-path whose number of internal vertices lying in $X$ is at most $k$. A set $X \subseteq V(G)$ is called a $k$-mutual visible set if every pair of distinct vertices in $X$ is $(X,k)$-visible. The $k$-mutual visibility number of a graph $G$, denoted by $\mu_k(G)$, is the maximum cardinality of a subset $X \subseteq V(G)$ that is $k$-mutual visible.  The $k$-visibility polynomial, $\mathcal{V}_k(G)$, of  $G$  is defined as
$\mathcal{V}_k(G)=\sum_{i\geq 0} r_{k,i} x^{i}$,  where $r_{k,i}$ denote  the number of $k$-mutual visibility sets in $G$ of cardinality $i$.
\end{defn}
Note that $\mu_0(G)$ coincides with the mutual-visibility number of $G$.
Moreover, the parameter $\mu_k(G)$ is non-decreasing in $k$.
Indeed, let $0 \le k \le \ell$ and let $X$ be a  $k$-mutual visible set in a graph $G$ of order $n$.
For any two distinct vertices $u,v \in X$, there exists a shortest $(u,v)$-path whose number
of internal vertices belonging to $X$ is at most $k$, and therefore at most $\ell$.
Hence, $X$ is also a $\ell$-mutual visible set, implying that $\mu_k(G) \le \mu_\ell(G)$.
In particular, $\mu_0(G) \le \mu_1(G) \le \mu_2(G) \le \cdots \le  |V(G)|$. Furthermore, if $G$ has order $n$, then $\mu_k(G)=\mu_{n-2}(G)$ for all $k \geq n-2$.

The concepts of total mutual visibility, outer mutual visibility, and dual mutual visibility were introduced by S. Cicerone et al in \cite{MV_10}. We extend these notions by defining the corresponding $k$-visibility variants, namely total $k$-mutual visibility, outer $k$-mutual visibility, and dual $k$-mutual visibility.
Let $G=(V,E)$ be a graph, let $X\subseteq V$, and let $k\ge 0$.
Write $\overline{X}=V\setminus X$.

\begin{itemize}
\item The set $X$ is called a \emph{total $k$-mutual visibility set} if every pair of distinct vertices $u,v\in V$ are $(X,k)$-visible.

\item The set $X$ is called an \emph{outer $k$-mutual visibility set} if every pair of distinct vertices $u,v\in X$ are $(X,k)$-visible and, moreover, every pair $u\in X$ and $v\in \overline{X}$ are $(X,k)$-visible.

\item The set $X$ is called a \emph{dual $k$-mutual visibility set} if every pair of distinct vertices $u,v\in X$ are $(X,k)$-visible and every pair of distinct vertices $u,v\in \overline{X}$ are $(X,k)$-visible.
\end{itemize}

When $k=0$, the above notions coincide with the total mutual-visibility set, outer mutual-visibility set, and dual mutual-visibility set introduced in~\cite{MV_10}. 

For a graph $G$ and an integer $k\ge 0$, we write 
$\mu_k^{\mathrm{tot}}(G)$, $\mu_k^{\mathrm{out}}(G)$, and 
$\mu_k^{\mathrm{dual}}(G)$ for the maximum cardinality of a total $k$-mutual visibility set, an outer $k$-mutual visibility set, and a dual $k$-mutual visibility set in $G$, respectively.
\begin{lemma}\label{P9.lem3}
Let $G$ be a graph, let $k\ge 0$, and let $H$ be a convex subgraph of $G$.
If $S\subseteq V(H)$ is a $k$-mutual visible set in $G$, then $S$ is a $k$-mutual visible set in $H$.
\end{lemma}
\begin{proof}
Let $u,v\in S$ be distinct.
Since $S$ is a  $k$-mutual visible set in $G$, there exists a shortest $(u,v)$-path $P$ in $G$ such that
the number of internal vertices of $P$ belonging to $S$ is at most $k$.
As $H$ is a convex subgraph of $G$ and $u,v\in V(H)$, every shortest $(u,v)$-path in $G$ is entirely contained in $H$.
In particular, the path $P$ is a shortest $(u,v)$-path in $H$.
Therefore, the number of internal vertices of $P$ belonging to $S$ is at most $k$ when considered in $H$ as well.
Hence, $S$ is a  $k$-mutual visible set in $H$.
\end{proof}
\begin{prop}\label{P9.prop4}
Let $G$ be a graph, let $k\ge 0$, and let $H$ be a convex subgraph of $G$.
Then $\mu_k(H)\le \mu_k(G)$.
\end{prop}
\begin{proof}
Let $S\subseteq V(H)$ be a  $k$-mutual visible set in $H$.
Since $H$ is a convex subgraph of $G$, every shortest path in $G$ between two vertices of $H$ is entirely contained in $H$.
Therefore, for any two distinct vertices $u,v\in S$, a shortest $(u,v)$-path witnessing the $k$-mutual visibility of $S$ in $H$
is also a shortest $(u,v)$-path in $G$.
Hence, the number of internal vertices of this path belonging to $S$ is at most $k$ in $G$ as well.
Thus, $S$ is a  $k$-mutual visible set in $G$, and the inequality $\mu_k(H)\le \mu_k(G)$ follows.
\end{proof}
\begin{prop}\label{P9.prop5}
Let $G=(V,E)$ be a graph, let $k\ge 0$, and let $V_1,\ldots,V_t\subseteq V$ satisfy
$\bigcup_{i=1}^t V_i=V$.
For each $i\in\{1,\ldots,t\}$, let $H_i=\mathrm{hull}(V_i)$.
Then $\mu_k(G)\le \sum_{i=1}^t \mu_k(H_i)$.
\end{prop}
\begin{proof}
Let $S\subseteq V(G)$ be a maximum  $k$-mutual visible set in $G$, so $|S|=\mu_k(G)$.
For each $i\in\{1,\ldots,t\}$, set $S_i =S\cap V(H_i)$.
Since $S_i\subseteq S$, the set $S_i$ is a  $k$-mutual visible set in $G$.
As $H_i$ is a convex subgraph of $G$, Lemma~\ref{P9.lem3} implies that
$S_i$ is a  $k$-mutual visible set in $H_i$, and hence $|S_i|\le \mu_k(H_i)$.
Since $S\subseteq \bigcup_{i=1}^t V(H_i)$, we have
$|S|=\bigl|\bigcup_{i=1}^t S_i\bigr|\le \sum_{i=1}^t |S_i|\le \sum_{i=1}^t \mu_k(H_i)$.
Thus $\mu_k(G)\le \sum_{i=1}^t \mu_k(H_i)$.
\end{proof}

\begin{prop}\label{P9.prop6}
Let $G$ be a graph and let $k\ge 0$.
Then $\mu_k(G)\ge gp(G)$ and $\mu_k(G)\ge \Delta(G)$.
Moreover, if $k\ge 1$, then $\mu_k(G)\ge \Delta(G)+1$.
\end{prop}

\begin{proof}
Let $S$ be a general position set of maximum cardinality.
By definition, for any two distinct vertices $u,v\in S$, no vertex of $S\setminus\{u,v\}$ lies on a shortest $(u,v)$-path.
Hence every such path has at most $0\le k$ internal vertices belonging to $S$, and therefore $S$ is a  $k$-mutual visible set.
This implies $\mu_k(G)\ge gp(G)$.

Next, let $v$ be a vertex of degree $\Delta(G)$.
For any two distinct vertices $x,y\in N(v)$, either $xy\in E(G)$ or $xy\notin E(G)$, in which case $(x,v,y)$ is a shortest $(x,y)$-path.
In both cases there exists a shortest $(x,y)$-path with no internal vertices in $N(v)$.
Thus $N(v)$ is a  $k$-mutual visible set, and hence $\mu_k(G)\ge |N(v)|=\Delta(G)$.

Finally, assume $k\ge 1$ and consider the closed neighborhood $N[v]$.
Let $x,y\in N[v]$ be distinct.
If $x$ and $y$ are adjacent, then they are trivially $k$-visible; otherwise, $(x,v,y)$ is a shortest $(x,y)$-path with exactly one internal vertex from $N[v]$, namely $v$.
It follows that $N[v]$ is a  $k$-mutual visible set for $k\ge 1$, and therefore $\mu_k(G)\ge |N[v]|=\Delta(G)+1$.
\end{proof}
\begin{theorem}\label{p9.th2}
Let $G$ be a graph of order $n$ and diameter $d$, and let $k\ge 0$. Then $\mu_k(G)\le n-d+k+1$.
\end{theorem}
\begin{proof}
Let $S\subseteq V(G)$ be a maximum  $k$-mutual visible set, so that $|S|=\mu_k(G)$.
Fix a diametral path $P=(x_0,x_1,\ldots,x_d)$ in $G$, where $d_G(x_0,x_d)=d$.
For all $0\le p<q\le d$ we have $d_G(x_p,x_q)=q-p$, since otherwise replacing the subpath
$x_p,\ldots,x_q$ by a shorter $(x_p,x_q)$-path would yield an $(x_0,x_d)$-path of length less than $d$,
contradicting the choice of $P$.

Let $ I=\{\,r\in\{0,\ldots,d\} : x_r\in S\,\}$.
If $|I|=0$, then all $d+1$ vertices of $P$ lie outside $S$, and hence
$n-|S|\ge d+1\ge d-1-k$.
If $|I|=1$, then at least $d$ vertices of $P$ lie outside $S$, implying
$n-|S|\ge d\ge d-1-k$.

Assume now that $|I|\ge 2$, and let $i=\min I$ and $j=\max I$.
By definition of $i$ and $j$, none of the vertices
$x_0,\ldots,x_{i-1},x_{j+1},\ldots,x_d$ belongs to $S$.
Thus at least $d-j+i$ vertices of $P$ lie in $V(G)\setminus S$. Since $x_i,x_j\in S$ and $S$ is  $k$-mutual visible, there exists a shortest $(x_i,x_j)$-path $Q$
whose number of internal vertices belonging to $S$ is at most $k$.
As $d_G(x_i,x_j)=j-i$, the path $Q$ has exactly $j-i-1$ internal vertices.
If $j-i\ge k+1$, then at least $j-i-1-k$ internal vertices of $Q$ lie in $V(G)\setminus S$.
Together with the $d-j+i$ vertices of $P$ outside $S$, this yields $
n-|S|\ge (d-j+i)+(j-i-1-k)=d-1-k$.
If $j-i<k+1$, then the vertices of $P$ outside $S$ already give $
n-|S|\ge d-j+i\ge d-(k+1)=d-1-k$.

In all cases, $n-|S|\ge d-1-k$, and hence $|S|\le n-d+k+1$.
Since $|S|=\mu_k(G)$, the result follows.
\end{proof}
\begin{theorem}\label{P9.th3}
Let $G$ be a graph with diameter $d$.
Then $\mu_k(G)=|V(G)|$ if and only if $k\ge d-1$.
\end{theorem}
\begin{proof}
Suppose that $k\ge d-1$.
Then every shortest path in $G$ has at most $d-1\le k$ internal vertices.
Hence $V(G)$ itself is a  $k$-mutual visible set, and therefore $\mu_k(G)=|V(G)|$.

Conversely, assume that $\mu_k(G)=|V(G)|$.
Let $u,v$ be a pair of vertices with $d_G(u,v)=d$.
Every shortest $(u,v)$-path has exactly $d-1$ internal vertices, all of which belong to $V(G)$.
Since $V(G)$ is  $k$-mutual visible, it follows that $d-1\le k$.
\end{proof}
\begin{theorem}\label{P9.th4}
Let $G$ be a graph of order $n$ and girth $g$. Then, for every integer $k\ge 0$, $\mu_k(G)\le n-g+2k+3$.
\end{theorem}
\begin{proof}
Suppose that $
\mu_k(G) > n-g+(2k+4)$. Let $S$ be a maximum  $k$-mutual visible set in $G$, and let $
C=(x_0,x_1,\ldots,x_{g-1})$
be a cycle of length $g$. Since $|S|>n-g+2k+3$, at most $g-(2k+4)$ vertices of $C$ lie outside $S$.
Hence $C$ contains at least $2k+4$ vertices of $S$. 
Choose vertices $
x_{i_1},x_{i_2},\ldots,x_{i_{2k+4}}\in S \cap V(C)$
with $ i_1<i_2<\cdots<i_{2k+4}$. Consider the vertices $x_{i_1}$ and $x_{i_{k+3}}$.
Along the cycle $C$, both $(x_{i_1},x_{i_{k+3}})$-paths contain at least
$k+1$ internal vertices belonging to $S$.

Since $S$ is a  $k$-mutual visible set, there exists a shortest
$(x_{i_1},x_{i_{k+3}})$-path $Q$ in $G$ such that at most $k$ internal vertices
of $Q$ belong to $S$. Moreover, the length of $Q$ is at most the distance between $x_{i_1}$ and
$x_{i_{k+3}}$ along the cycle $C$. Consequently, the union of $Q$ with the shorter of the two
$(x_{i_1},x_{i_{k+3}})$-paths on $C$ contains a cycle of length at most $g$.
By the definition of girth, this cycle must have length exactly $g$. Thus, the resulting cycle of length $g$ contains at least one fewer vertex of $S$ than $C$, since the removed segment of $C$ contains at least $k+1$ vertices of $S$, whereas $Q$ contributes at most $k$ internal vertices of $S$.
Repeating this replacement process finitely many times, we obtain a cycle of length
$g$ containing at most $2k+3$ vertices of $S$.
Therefore the result follows.
\end{proof}
\begin{remark}
Let $G$ be a graph of order $n$, diameter $d$, and girth $g$, and let $k\ge 0$.
By Theorems~\ref{p9.th2} and~\ref{P9.th4}, we have
$\mu_k(G)\leq \min\{ n-d+k+1,  n-g+2k+3, n \}$.
\end{remark}
\begin{theorem}\label{P9.th6}
Let $G$ be a graph of order $n$, and let $k\ge 0$.
Suppose that $G$ contains an isometric path $P=(x_0, x_1, \ldots, x_\ell)$.
Then $\mu_k(G)\le n-\ell+k+1$.
\end{theorem}

\begin{proof}
Let $S\subseteq V(G)$ be a  $k$-mutual visible set with $|S|=\mu_k(G)$ and let
$I=\{\, i\in\{0,1,\dots,\ell\}: x_i\in S \,\}$.
If $|I|\le 1$, then at least $\ell$ vertices of $P$ lie in $V(G)\setminus S$, so
$|S|\le n-\ell\le n-\ell+k+1$.
Assume $|I|\ge 2$, and let $i=\min I$, $j=\max I$.
Then $x_0,\dots,x_{i-1},x_{j+1},\dots,x_\ell\in V(G)\setminus S$, so at least $\ell-(j-i)$ vertices of $P$ lie outside $S$.
Since $S$ is  $k$-mutual visible, there exists a shortest $(x_i,x_j)$-path $Q$ containing at most $k$ internal vertices from $S$.
Because $P$ is isometric, $d_G(x_i,x_j)=j-i$, hence every shortest $(x_i,x_j)$-path has exactly $j-i-1$ internal vertices, and therefore at least $j-i-1-k$ of them belong to $V(G)\setminus S$.
Thus $|V(G)\setminus S|\ge (\ell-(j-i))+(j-i-1-k)=\ell-1-k$, and consequently
$|S|=n-|V(G)\setminus S|\le n-(\ell-1-k)=n-\ell+k+1$.
\end{proof}
\section{$k$-mutual visibility for some graph classes}
\begin{prop}\label{P9.prop1}
Let $P_n$ be the path on $n\ge 1$ vertices and let $k\ge 0$.
Then $ \mu_k(P_n)=min \{n, k+2\}$.
\end{prop}

\begin{proof}
Let $X\subseteq V(P_n)$ and list the vertices of $X$ in their natural order along the path as $x_1,\ldots,x_m$, where $m=|X|$. Since $P_n$ is a path, for any $1\le i<j\le m$ the unique shortest $(x_i,x_j)$-path is the subpath between $x_i$ and $x_j$, whose internal vertices from $X$ are exactly $x_{i+1},\ldots,x_{j-1}$; hence their number is $j-i-1$. Thus $X$ is  $k$-mutual visible if and only if $j-i-1\le k$ for all $1\le i<j\le m$. The maximum value of $j-i-1$ occurs when $i=1$ and $j=m$, yielding $m-2\le k$, and therefore $m\le k+2$. Consequently, $\mu_k(P_n)\le \min\{n,k+2\}$. 

For the reverse inequality, if $k+2\le n$, choose any $k+2$ vertices of $P_n$. Then for every pair $x_i,x_j$ in this set, the number of internal vertices of the unique
$(x_i, x_j)$-subpath  is at most $k$, so the set is  $k$-mutual visible and $\mu_k(P_n)\ge k+2$. If $k+2>n$, then for every pair of vertices of $P_n$ the number of internal vertices on their unique shortest path is at most $n-2<k$, so $V(P_n)$ itself is  $k$-mutual visible and $\mu_k(P_n)=n$. Hence $\mu_k(P_n)=\min\{n,k+2\}$.
\end{proof}
\begin{remark}
The bound in Theorem~\ref{p9.th2} is tight.
For the path $P_n$ we have $d=\operatorname{diam}(P_n)=n-1$ and, by Proposition~\ref{P9.prop1},
$\mu_k(P_n)=k+2$ for all $0 \leq k \leq n-2$, yielding equality.
\end{remark}
\begin{prop}\label{P9.prop2}
Let $n\ge 3$ and let $k$ be an integer with
$0\le k\le n-2$.
Then $ \mu_k(C_n)=\min\{ n, 2k+3 \}$.
\end{prop}
\begin{proof}
Let $X\subseteq V(C_n)$ be a  $k$-mutual visible set, and list the vertices of $X$
in cyclic order as $x_1,x_2,\ldots,x_m$, where $m=|X|$.

Suppose that $m\ge 2k+4$.
Consider the vertices $x_1$ and $x_{k+3}$.
Along one $(x_1,x_{k+3})$-path of the cycle, the internal vertices belonging to $X$
are $x_2,\ldots,x_{k+2}$, yielding $k+1$ vertices, while along the complementary arc
the number of internal vertices from $X$ is $m-2-(k+1)=m-k-3\ge k+1$, since $m\ge 2k+4$.
Thus every shortest $(x_1,x_{k+3})$-path contains at least $k+1$ internal vertices of $X$,
contradicting the $k$-mutual visibility of $X$.
Hence $m\le 2k+3$, and therefore $\mu_k(C_n)\le \min\{n,2k+3\}$.

Every shortest path in $C_n$ has at most $\lfloor n/2\rfloor-1$ internal vertices.
If $2k+3\ge n$, then $k\ge (n-3)/2$, and since $k$ is an integer it follows that
$k\ge \lceil (n-3)/2\rceil=\lfloor n/2\rfloor-1$.
Hence $V(C_n)$ itself is a  $k$-mutual visible set, and consequently $\mu_k(C_n)=n$.

Now assume that $2k+3<n$.
Set $r=n-(2k+3)$ and choose integers
$s=\lfloor r/2\rfloor$ and $t=\lceil r/2\rceil$, so that $s+t=r$ and $t-s\in\{0,1\}$.
Select two blocks of consecutive vertices on $C_n$, namely a block
$A=\{a_0,a_1,\ldots,a_{k+1}\}$ of size $k+2$ and a block
$B=\{b_0,b_1,\ldots,b_k\}$ of size $k+1$, and place them so that there are exactly
$s$ unselected vertices between $a_{k+1}$ and $b_0$ and exactly $t$ unselected vertices
between $b_k$ and $a_0$.
This is possible since $(k+2)+(k+1)+s+t=n$.
Let $X=A\cup B$, so that $|X|=2k+3$.

If $u,v\in A$ (or $u,v\in B$), then one $(u,v)$-path is the subpath induced by the
corresponding block, whose length is at most $k+1$, while the complementary $(u,v)$-path
traverses vertices outside the block and has length at least $n-k-1>k+1$ (since $2k+3 < n$).
Hence the shortest $(u,v)$-path is contained entirely in the corresponding block and has at most $k$ internal vertices from $X$.
Now let $u=a_i\in A$ and $v=b_j\in B$, and denote by $P_s$ and $P_t$ the two $(u,v)$-paths
passing through the gaps of sizes $s$ and $t$, respectively.
The numbers of internal vertices from $X$ on these paths are
$c_s=(k+1-i)+j$ and $c_t=i+(k-j)$, respectively, while their lengths are
$\ell_s=(k+1-i)+(s+1)+j$ and $\ell_t=i+(t+1)+(k-j)$.
Note that $c_s\le k$ if and only if $i>j$, whereas $c_t\le k$ if and only if $i\le j$.

If $i>j$, then $j-i\le -1$, and hence
$\ell_s-\ell_t=1+s-t+2(j-i)<0$, so $P_s$ is the unique shortest $(u,v)$-path and
$c_s\le k$.
If $i\le j$, then $\ell_s-\ell_t\ge 0$, so $P_t$ is a shortest $(u,v)$-path and
$c_t\le k$.
Therefore, in all cases, there exists a shortest $(u,v)$-path containing at most $k$
internal vertices from $X$.
It follows that $X$ is a  $k$-mutual visible set, and hence $\mu_k(C_n)\ge 2k+3$.

Combining the upper and lower bounds yields
$\mu_k(C_n)=\min\{n,2k+3\}$.

\end{proof}
\begin{remark}
The bound in Theorem~\ref{P9.th4} is tight.
For the cycle $C_n$ we have $g=n$, and by Proposition~\ref{P9.prop2}$,
\mu_k(C_n)=\min\{n,2k+3\}$.
When $2k+3\le n$, it follows that
$\mu_k(C_n)=2k+3=n-g+(2k+3)$.
\end{remark}

\begin{prop}\label{P9.prop3}
Let $m,n\ge 1$. Then, for every $k\ge 1$, $\mu_k(K_{m,n})=m+n$.
\end{prop}

\begin{proof}
Let $K_{m,n}$ be the complete bipartite graph with bipartition $(A,B)$, where $|A|=m$ and $|B|=n$, and set $X =V(K_{m,n})=A\cup B$. For any two distinct vertices $u,v\in X$, if $u$ and $v$ lie in different parts, then $uv\in E(K_{m,n})$ and a shortest $(u,v)$-path has no internal vertices, while if $u$ and $v$ lie in the same part, then $d(u,v)=2$ and every shortest $(u,v)$-path has exactly one internal vertex. Hence, for every pair $u,v\in X$, there exists a shortest $(u,v)$-path with at most $1\le k$ internal vertex, since $k\ge 1$. Thus $X$ is a  $k$-mutual visible set, implying $\mu_k(K_{m,n})\ge |X|=m+n$. The reverse inequality $\mu_k(K_{m,n})\le |V(K_{m,n})|=m+n$ is immediate, and therefore $\mu_k(K_{m,n})=m+n$.
\end{proof}
\begin{theorem}\label{P9.th8}
Let $G$ be a graph of order $n$. Then, for every $k\ge 2$, $\mu_k(G\circ K_m)=mn+\mu_{k-2}(G)$.
\end{theorem}

\begin{proof}
For each vertex $v\in V(G)$, let $K_m^v$ denote the copy of $K_m$
attached to $v$ in $G\circ K_m$, and define $B_v=\{v\}\cup V(K_m^v)$.
Then $B_v$ induces a complete graph of order $m+1$.

Let $k\ge 2$. Let $X$ be a maximum $(k-2)$-mutual visible set of
$G$. Define
\[
S=\left(\bigcup_{v\in V(G)}V(K_m^v)\right)\cup X.
\]
Then $|S|=mn+\mu_{k-2}(G)$. We claim that $S$ is a $k$-mutual visible set of $G\circ K_m$. Let $a,b\in S$ be distinct. If $a$ and $b$ belong to the same block
$B_v$, then they are adjacent. Otherwise, suppose that $a$ lies in
$B_u$ and $b$ lies in $B_v$, where $u\ne v$. A shortest $(a,b)$-path in
$G\circ K_m$ is obtained from a shortest $(u,v)$-path in $G$ by adding the necessary edges inside $B_u$ and $B_v$. Since $X$ is mutual $(k-2)$-visible in $G$, the shortest $(u,v)$-path in
$G$ can be chosen so that it contains at most $k$ vertices from $X$ in total, counting possible endpoints. Hence the corresponding shortest $(a,b)$-path in $G\circ K_m$ contains at most $k$ internal vertices from $S$. Therefore $S$ is $k$-mutual visible, and so $\mu_k(G\circ K_m)\ge mn+\mu_{k-2}(G)$.

Conversely, let $S$ be a maximum $k$-mutual visible set of $G\circ K_m$. By \eqref{P9.eq2}, $|S|\leq mn+|D|$, where $D=\{v\in V(G): B_v\subseteq S\}$. We claim that $D$ is a $(k-2)$-mutual visible set of $G$. Let $u,v\in D$ be distinct. Since $B_u,B_v\subseteq S$, choose vertices $a\in V(K_m^u)$ and $b\in V(K_m^v)$. Then $a,b\in S$. Since $S$ is $k$-mutual visible, there exists a shortest $(a,b)$-path in
$G\circ K_m$ containing at most $k$ internal vertices from $S$. Every shortest $(a,b)$-path has the form
$(a, P, b)$ where $P$ is a shortest $(u,v)$-path in $G$. The vertices $u$ and $v$ are internal vertices of this path and both belong to $S$. Thus they already contribute two internal vertices from $S$. Therefore the internal part of $P$ contains at most $k-2$ vertices from $S$ hence from $D$. Hence
$u$ and $v$ are $(D,k-2)$-visible in $G$. Thus $D$ is a $(k-2)$-mutual
visible set of $G$, and so
$|D|\le \mu_{k-2}(G)$. Consequently,
$|S|\le mn+\mu_{k-2}(G)$. Combining the lower and upper bounds gives
$\mu_k(G\circ K_m)=mn+\mu_{k-2}(G)$
for every $k\ge 2$.
\end{proof}
\begin{coro}
Let $G_0=G$ and $
G_i=G_{i-1}\circ K_1,\ 1\le i\le \left\lfloor \frac{k}{2}\right\rfloor.$
Then
\[
\mu_k\left(G_{\left\lfloor k/2\right\rfloor}\right)
=
\left(2^{\left\lfloor k/2\right\rfloor}-1\right)n+
\begin{cases}
\mu(G), & \text{if } k \text{ is even},\\[2mm]
\mu_1(G), & \text{if } k \text{ is odd}.
\end{cases}
\]
\end{coro}

\begin{proof}
Let $ t=\left\lfloor \frac{k}{2}\right\rfloor$. Then $k=2t+\varepsilon$, where $\varepsilon\in\{0,1\}$. Since $|V(G_i)|=2^i n$ for each $i\ge 0$, repeated application of Theorem~\ref{P9.th8} gives

\[
\begin{aligned}
\mu_k(G_t)
&=\mu_{2t+\varepsilon}(G_t)\\
&=|V(G_{t-1})|+\mu_{2t+\varepsilon-2}(G_{t-1})\\
&=|V(G_{t-1})|+|V(G_{t-2})|
  +\mu_{2t+\varepsilon-4}(G_{t-2})\\
&\hspace{1cm}\vdots\\
&=\sum_{i=0}^{t-1}|V(G_i)|+\mu_\varepsilon(G)\\
&=\sum_{i=0}^{t-1}2^i n+\mu_\varepsilon(G)=(2^t-1)n+\mu_\varepsilon(G).
\end{aligned}
\]
If $k$ is even, then $\varepsilon=0$ and
$\mu_\varepsilon(G)=\mu(G)$, whereas if $k$ is odd, then
$\varepsilon=1$ and $\mu_\varepsilon(G)=\mu_1(G)$.
The result follows.
\end{proof}
\section{$k$-Mutual visibility of block graphs}
\begin{defn}
Let $G$ be a block graph, and let $H$ be its block--cutpoint tree with bipartition
$
V(H)=A \cup \mathcal{B}'$,
where $A$ is the set of articulation vertices of $G$ and $\mathcal{B}'$ is the set of vertices corresponding to the blocks of $G$.
A subset $Z\subseteq V(H)$ is $k$-admissible if, for every pair of distinct vertices
$\alpha,\beta\in Z$, $
\bigl|\bigl(V(P_H(\alpha,\beta))\setminus\{\alpha,\beta\}\bigr)\cap (Z\cap A)\bigr|\le k$,
where $P_H(\alpha,\beta)$ denotes the unique $(\alpha, \beta)$-path in the tree $H$.
\end{defn}
Intuitively, a $k$-admissible set $Z$ is one in which no two selected vertices in the
block--cutpoint tree are separated by more than $k$ selected articulation vertices.
Since articulation vertices represent mandatory passage points between blocks in a
block graph, the definition bounds the number of such obstructions along the unique
path connecting any pair of vertices of $Z$.
\begin{lemma}\label{P9.lem1}
Let $G$ be a block graph, and let $H$ be its block--cutpoint tree with bipartition
$V(H)=A \cup \mathcal{B}'$,
where $A$ is the set of articulation vertices of $G$ and $\mathcal{B}'$ is the set of vertices corresponding to the blocks of $G$.
If $Z\subseteq V(H)$ is $k$-admissible, then
$X_Z=\Big(\bigcup_{b\in Z\cap \mathcal{B}'} (V(B_b)\setminus A)\Big)\cup (Z\cap A)$
is a  $k$-mutual visible set in $G$, where $B_b$ denotes the block of $G$ corresponding to the vertex $b\in\mathcal{B}'$.
\end{lemma}
\begin{proof}
Let $x,y\in X_Z$ be distinct. Define a map $\varphi:X_Z\to V(H)$ by
\[
\varphi(x)=
\begin{cases}
x, & \text{if } x\in A,\\
b, & \text{if } x\in V(B_b)\setminus A \text{ for some } b\in\mathcal{B}'.
\end{cases}
\]
This map is well defined, since every vertex of $G$ that is not an articulation vertex belongs to a unique block.

Set $\alpha=\varphi(x)$ and $\beta=\varphi(y)$. Let $P_H(\alpha,\beta)$ be the unique $(\alpha,\beta)$-path in the tree $H$.
Let $a_0,a_1,\dots,a_t$ be the articulation vertices on $P_H(\alpha,\beta)$, listed in the order in which they appear along the path from $\alpha$.
If $t\ge 1$, then for each $i\in\{0,1,\dots,t-1\}$, the vertices $a_i$ and $a_{i+1}$ are separated on $P_H(\alpha,\beta)$ by exactly one
block-vertex.

We construct an $(x,y)$-path $P$ in $G$ as follows. If $x\notin A$, then $x\in V(B_\alpha)\setminus A$ and
$a_0\in V(B_\alpha)\cap A$, so $xa_0\in E(G)$ since $B_\alpha$ is a clique; if $x\in A$, we start at $x=a_0$.
For each $i\in\{0,1,\dots,t-1\}$, the vertices $a_i$ and $a_{i+1}$ lie in a common block; since every block of a block graph
is a clique, we have $a_i a_{i+1}\in E(G)$. Finally, if $y\notin A$, then $a_t\in V(B_\beta)\cap A$ and
$y\in V(B_\beta)\setminus A$, so $a_t y\in E(G)$; if $y\in A$, we end at $y=a_t$.
Concatenating these edges yields an $(x,y)$-path $P$.

We claim that $P$ is a shortest $(x,y)$-path in $G$. Let $Q$ be any $(x,y)$-path in $G$.
Consider the sequence of blocks of $G$ met by $Q$. Whenever $Q$ moves from one block to another, it must pass through
their common articulation vertex. Hence $Q$ determines a walk in the block--cutpoint tree $H$ from $\alpha$ to $\beta$.
Since $H$ is a tree, this walk contains the unique $(\alpha,\beta)$-path $P_H(\alpha,\beta)$ as a subwalk.
In particular, $Q$ must meet the articulation vertices $a_0,a_1,\dots,a_t$ in this order.

We now compare lengths. If $x\notin A$, then any $(x,y)$-path must enter the articulation set $A$ through a vertex of
$B_\alpha\cap A$, and thus uses at least one edge before reaching $a_0$; the path $P$ uses exactly one edge $xa_0$.
Similarly, if $y\notin A$, then any $(x,y)$-path must use at least one edge after leaving $a_t$; the path $P$ uses exactly one
edge $a_t y$. Moreover, for each $i\in\{0,1,\dots,t-1\}$, any $(x,y)$-path must pass through $a_i$ and $a_{i+1}$ in this order.
Since $a_i$ and $a_{i+1}$ lie in a common block and every block of a block graph is a clique,
we have $a_i a_{i+1}\in E(G)$.
Therefore, any $(x,y)$-path has length at least the number of these forced transitions, together with the possible initial
and final edges, and the path $P$ attains this lower bound. Thus $P$ is a shortest $(x,y)$-path in $G$.

Finally, the only internal vertices of $P$ that can lie in $X_Z$ are articulation vertices in $Z\cap A$ that occur internally
on $P_H(\alpha,\beta)$. Hence
\[
\bigl|\bigl(V(P)\setminus\{x,y\}\bigr)\cap X_Z\bigr|
=
\bigl|\bigl(V(P_H(\alpha,\beta))\setminus\{\alpha,\beta\}\bigr)\cap (Z\cap A)\bigr|
\le k,
\]
where the inequality holds because $Z$ is $k$-admissible.
Thus $x$ and $y$ are $(X_Z,k)$-visible. Since $x,y\in X_Z$ were arbitrary, $X_Z$ is a  $k$-mutual visible set in $G$.
\end{proof}

\begin{lemma}\label{P9.lem2}
Let $G$ be a block graph, and let $H$ be its block--cutpoint tree with bipartition
$V(H)=A \cup \mathcal{B}'$,
where $A$ is the set of articulation vertices of $G$ and $\mathcal{B}'$ is the set of vertices corresponding to the blocks of $G$.
If $X\subseteq V(G)$ is a  $k$-mutual visible set, then
$Z=(X\cap A)\cup\{\, b\in \mathcal{B}': X\cap (V(B_b)\setminus A)\neq\emptyset \,\}$
is $k$-admissible, where $B_b$ denotes the block of $G$ corresponding to the vertex $b\in\mathcal{B}'$.
\end{lemma}
\begin{proof}
Let $\alpha,\beta\in Z$ be distinct. By definition, $Z$ is the disjoint union of $X\cap A$ and
$\{\, b\in \mathcal{B}' : X\cap (V(B_b)\setminus A)\neq\emptyset \,\}$.
Choose vertices $x_\alpha,x_\beta\in X$ as follows: if $\alpha\in X\cap A$, then set $x_\alpha=\alpha$; if $\alpha\in\mathcal{B}'$,
then $X\cap (V(B_\alpha)\setminus A)\neq\emptyset$ by the definition of $Z$, and we choose $x_\alpha$ to be any vertex in
$X\cap (V(B_\alpha)\setminus A)$. Define $x_\beta$ analogously.

Let $\varphi:V(G)\to V(H)$ be given by $\varphi(v)=v$ if $v\in A$, and $\varphi(v)=b$ if $v\in V(B_b)\setminus A$.
Then $\varphi(x_\alpha)=\alpha$ and $\varphi(x_\beta)=\beta$. Since $X$ is a  $k$-mutual visible set, there exists a shortest
$(x_\alpha,x_\beta)$-path $P$ in $G$ such that $|(V(P)\setminus\{x_\alpha,x_\beta\})\cap X|\le k$.

Let $P_H(\alpha,\beta)$ be the unique $(\alpha,\beta)$-path in the tree $H$, and set
$S=(V(P_H(\alpha,\beta))\setminus\{\alpha,\beta\})\cap (Z\cap A)$.
We claim that
$S\subseteq (V(P)\setminus\{x_\alpha,x_\beta\})\cap X$.
Indeed, let $a\in S$. By the definition of $S$, the vertex $a$ is an articulation vertex of $G$ that lies internally on
$P_H(\alpha,\beta)$, and hence $\alpha$ and $\beta$ lie in different components of $H-a$.
By the definition of the block--cutpoint tree, it follows that $x_\alpha$ and $x_\beta$ lie in different components of $G-a$.
Consequently, every $(x_\alpha,x_\beta)$-path in $G$ contains $a$, and in particular $a\in V(P)$. Moreover, $a\neq x_\alpha$ and $a\neq x_\beta$. Indeed, if $\alpha\in A$, then $x_\alpha=\alpha\neq a$, while if
$\alpha\in\mathcal{B}'$, then $x_\alpha\in V(B_\alpha)\setminus A$ whereas $a\in A$; the same argument applies to $x_\beta$.
Finally, since $a\in S\subseteq Z\cap A$ and $Z\cap A=X\cap A$ by definition of $Z$, we have $a\in X$.
Thus $a\in (V(P)\setminus\{x_\alpha,x_\beta\})\cap X$, proving the claim.

Therefore, $ |S|\le |(V(P)\setminus\{x_\alpha,x_\beta\})\cap X|\le k$.
Since $\alpha,\beta\in Z$ were arbitrary, it follows that $Z$ is $k$-admissible.
\end{proof}
\begin{theorem}\label{P9.th1}
Let $G$ be a block graph, and let $H$ be its block--cutpoint tree with bipartition
$V(H)=A \cup \mathcal{B}'$,
where $A$ is the set of articulation vertices of $G$ and $\mathcal{B}'$ is the set of vertices corresponding to the blocks of $G$.
For a $k$-admissible set $Z\subseteq V(H)$, define $X_Z=\Big(\bigcup_{b\in Z\cap \mathcal{B}'} (V(B_b)\setminus A)\Big)\cup (Z\cap A)$, where $B_b$ denotes the block of $G$ corresponding to the vertex $b\in\mathcal{B}'$.
Then $ \mu_k(G)=\max\{ |X_Z| : Z\subseteq V(H)\text{ is $k$-admissible} \}$.
\end{theorem}
\begin{proof}
Let $M=\max\{ |X_Z|: Z\subseteq V(H)\text{ is $k$-admissible} \}$.
First we show that $M\le \mu_k(G)$. Let $Z\subseteq V(H)$ be $k$-admissible. By Lemma~\ref{P9.lem1}, the set $X_Z$ is a mutual
$k$-visible set in $G$. Hence $|X_Z|\le \mu_k(G)$ for every $k$-admissible $Z$, and therefore $M\le \mu_k(G)$.

For the reverse inequality, let $X\subseteq V(G)$ be a  $k$-mutual visible set with $|X|=\mu_k(G)$. Define $
Z=(X\cap A)\cup\{\,b\in\mathcal{B}': X\cap (V(B_b)\setminus A)\neq\emptyset\,\}$.
By Lemma~\ref{P9.lem2}, the set $Z$ is $k$-admissible. Moreover, $X\subseteq X_Z$. Indeed, if $x\in X\cap A$, then $x\in Z\cap A\subseteq X_Z$.
If $x\in X\setminus A$, then $x\in V(B_b)\setminus A$ for a unique $b\in\mathcal{B}'$.
Since $X\cap (V(B_b)\setminus A)\neq\emptyset$, we have $b\in Z$, and hence $b\in Z\cap\mathcal{B}'$.
Therefore $x\in V(B_b)\setminus A\subseteq X_Z$.
Consequently, $\mu_k(G)=|X|\le |X_Z|\le M$,
and thus $\mu_k(G)\le M$. Combining the two inequalities yields $\mu_k(G)=M$, as required.
\end{proof}
\begin{coro}\label{P9.cor1}
Let $G$ be a block graph, and let $H$ be its block--cutpoint tree with bipartition
$V(H)=A\cup\mathcal{B}'$.
If $Z\subseteq V(H)$ is $0$-admissible, then
$|X_Z|\le |V(G)\setminus A|$.
\end{coro}

\begin{proof}
By Theorem~\ref{P9.th1},
\begin{equation}\label{P9.eq1}
\mu(G)=\mu_0(G)=\max\{\,|X_Z| : Z\subseteq V(H)\text{ is $0$-admissible}\,\}.
\end{equation}
Di Stefano~\cite{Stefano} proved that $\mu(G)=|V(G)\setminus A|$.
Since $\mathcal{B}'\cap A=\emptyset$, the set $\mathcal{B}'$ is $0$-admissible. Moreover,
\[
X_{\mathcal{B}'}=
\Big(\bigcup_{b\in \mathcal{B}'} (V(B_b)\setminus A)\Big)\cup(\mathcal{B}'\cap A)
=\bigcup_{b\in \mathcal{B}'} (V(B_b)\setminus A).
\]
In a block graph, every vertex is either an articulation vertex or a non-articulation vertex belonging to a unique block.
Consequently, the sets $V(B_b)\setminus A$ are pairwise disjoint, and their union equals $V(G)\setminus A$.
Thus $
X_{\mathcal{B}'}=V(G)\setminus A $, and 
$|X_{\mathcal{B}'}|=|V(G)\setminus A|=\mu(G)$. Therefore, the maximum in~\eqref{P9.eq1} is attained by the $0$-admissible set $Z=\mathcal{B}'$.
In particular, for every $0$-admissible set $Z\subseteq V(H)$ we have $|X_Z|\le |X_{\mathcal{B}'}|=|V(G)\setminus A|$.
\end{proof}
\section{Computational Complexity}
The following problem, which asks to test whether a given set of points is a $k$-mutual
visibility set, can be solved in polynomial time.
\begin{defn}{K-MUTUAL VISIBILITY TEST:}
\leavevmode\par
 Instance: A graph $G=(V,E)$, a set of points $S\subseteq V$, and an integer $k\ge 0$.
 
 Question: Is $S$ a $k$-mutual visibility set of $G ?$ 
\end{defn}
The solution is provided by means of Algorithm \textsc{kMV}, which in turn uses
Procedure \textsc{BFS\_kMV} as a sub-routine. Procedure \textsc{BFS\_kMV} and Algorithm \textsc{kMV} are given below.\\
{
\renewcommand{\thealgocf}{}
\begin{algorithm}[H]
\caption{\textsc{kMV}}
\KwIn{A graph $G=(V,E)$, a set of points $S\subseteq V$, an integer $k\ge 0$}
\KwOut{True if $S$ is a $k$-mutual visibility set, False otherwise}

\If{points in $S$ are in different connected components of $G$}{
\Return False\;
}
Let $H$ be the connected component of $G$ containing $S$\;

\ForEach{$v\in S$}{
$DP\leftarrow \textsc{BFS\_kMV}(H,S,v)$\;
\ForEach{$q\in S\setminus\{v\}$}{
\If{$DP[q]>k+1$}{
\Return False\;
}
}
}
\Return True\;
\end{algorithm}
}  
\begin{theorem}\label{P9.th5}
Let $G=(V,E)$ be a graph, let $S\subseteq V$, and let $k\ge 0$.
Algorithm \textsc{kMV} decides \textsc{$k$-Mutual Visibility Test} in time $O\bigl(|S|(|V|+|E|)+|S|^2\bigr)$, and hence in polynomial time.
\end{theorem}
\begin{proof}
We first estimate the running time of the subroutine \textsc{BFS\_kMV} on a connected graph $H=(V(H),E(H))$ with input $(H,S,v)$, where $S\subseteq V(H)$ and $v\in S$.

The initialization of the arrays $\mathrm{dist}[\cdot]$ and $\mathrm{cnt}[\cdot]$ (Lines~1--4) requires $O(|V(H)|)$ time, while initializing $DP[s]$ for $s\in S$ (Lines~5--7) requires $O(|S|)$ time.
The breadth-first search phase (Lines~10--15) visits each vertex at most once and scans each adjacency list once; hence every edge of $H$ is examined at most twice.
Therefore this phase runs in $O(|V(H)|+|E(H)|)$ time.

After the distances have been computed, the vertices are naturally partitioned into BFS levels according to their distance from $v$.
Thus they can be processed in nondecreasing order of $\mathrm{dist}$ without sorting.
In the second phase (Lines~17--24), for each vertex $u$ the procedure scans all neighbors $w\in N_H(u)$ and performs constant-time updates whenever $\mathrm{dist}[w]=\mathrm{dist}[u]+1$.
Consequently, each edge is examined a constant number of times, contributing $O(|E(H)|)$ time.
Finally, copying $\mathrm{cnt}[s]$ to $DP[s]$ (Lines~25--26) requires $O(|S|)$ time.
Since each vertex is processed a constant number of times and each edge
is scanned a constant number of times in both phases,
the total running time of \textsc{BFS\_kMV} is $O(|V(H)|+|E(H)|)$.

We now analyze Algorithm \textsc{kMV}.
The preliminary connectivity test (Line~1) can be performed by a single BFS in $O(|V|+|E|)$ time.
For each $v\in S$ (outer loop), the algorithm invokes $\textsc{BFS\_kMV}(H,S,v)$ once, at cost $O(|V|+|E|)$, and then verifies the condition $DP[q]\le k+1$ for all $q\in S\setminus\{v\}$ (inner loop), which requires $O(|S|)$ time per $v$.
Therefore, the total running time of \textsc{kMV} is
$O(|V|+|E|) + \sum_{v\in S}\bigl(O(|V|+|E|)+O(|S|)\bigr)
= O(|V|+|E|) + |S|\bigl(O(|V|+|E|)+O(|S|)\bigr)= O\bigl(|S|(|V|+|E|)+|S|^2\bigr)$.
Since $|S|\le |V|$, this bound is polynomial in the input size.
Hence \textsc{kMV} solves \textsc{$k$-Mutual Visibility Test} in polynomial time.
\end{proof}
{
\renewcommand{\thealgocf}{}   
\renewcommand{\algorithmcfname}{Procedure}
\begin{algorithm}
\caption{\textsc{BFS\_kMV}}
\KwIn{A connected graph $G=(V,E)$, a set $S\subseteq V$, a vertex $v\in V$}
\KwOut{For each $s\in S$, the minimum number of vertices of $S\setminus\{v\}$
on a shortest $(v,s)$-path}

\ForEach{$u\in V$}{
$\mathrm{dist}[u]\leftarrow \infty$\;
$\mathrm{cnt}[u]\leftarrow \infty$\;
}
\ForEach{$s\in S$}{
$DP[s]\leftarrow \infty$\;
}
$\mathrm{dist}[v]\leftarrow 0$\;
$\mathrm{cnt}[v]\leftarrow 0$\;
Let $Q$ be a queue\;
$Q.\mathrm{enqueue}(v)$\;

\While{$Q$ is not empty}{
$u\leftarrow Q.\mathrm{dequeue}()$\;
\ForEach{$w\in N_G(u)$}{
\If{$\mathrm{dist}[w]=\infty$}{
$\mathrm{dist}[w]\leftarrow \mathrm{dist}[u]+1$\;
$Q.\mathrm{enqueue}(w)$\;
}
}
}

Process vertices $u\in V$ in nondecreasing order of $\mathrm{dist}[u]$\;
\ForEach{$u$ in this order}{
\If{$\mathrm{cnt}[u]<\infty$}{
\ForEach{$w\in N_G(u)$}{
\If{$\mathrm{dist}[w]=\mathrm{dist}[u]+1$}{
\If{$w\in S$ and $w\neq v$}{
$\mathrm{cnt}[w]\leftarrow \min\{\mathrm{cnt}[w],\mathrm{cnt}[u]+1\}$\;
}
\Else{
$\mathrm{cnt}[w]\leftarrow \min\{\mathrm{cnt}[w],\mathrm{cnt}[u]\}$\;
}
}
}
}
}

\ForEach{$s\in S$}{
$DP[s]\leftarrow \mathrm{cnt}[s]$\;
}
\Return $DP$\;
 \end{algorithm}
}
\begin{theorem}\label{P9.th7}
Let $G$ be a graph of order $n$, and let $r\ge 1$.
Let $G'$ be obtained from $G$ by attaching a path of length $r-1$
to each vertex of $G$ by an edge. Then, for $\varepsilon\in\{0,1\}$,
$\mu_{2r+\varepsilon}(G')=rn+\mu_\varepsilon(G)
$.
\end{theorem}
\begin{proof}
For each vertex $v\in V(G)$, let $P_v=(v=v_0,v_1,\ldots,v_r)$ denote the path attached at $v$, and let $L=\{v_i:v\in V(G),\,1\le i\le r\}$ be the set of all newly added vertices. Then $|L|=rn$.

Let $X$ be a maximum $\varepsilon$-mutual visible set of $G$, where $\varepsilon\in\{0,1\}$. We claim that $S=L\cup X$ is a $(2r+\varepsilon)$-mutual visible set of $G'$. Indeed, if two vertices of $S$ lie on the same attached path $P_v$, then the unique shortest path between them contains at most $r-1$ internal vertices of $S$. Otherwise, a shortest path between them consists of a path in $G$ together with portions of at most two attached paths ($P_v , P_u$). Since $X$ is $\varepsilon$-mutual visible in $G$, the segment in $G$ may be chosen to contain at most $\varepsilon+2$ vertices of $X$, while the two attached path portions contribute at most $r-1$ internal vertices each. Hence the path contains at most $2(r-1)+\varepsilon+2= 2r+\varepsilon$ internal vertices of $S$. Therefore,
$\mu_{2r+\varepsilon}(G')\ge |S|=rn+\mu_\varepsilon(G)$.

For the reverse inequality, let $S$ be a maximum $(2r+\varepsilon)$-mutual visible set of $G'$. For each $v\in V(G)$, define $B_v=\{v_0,v_1,\ldots,v_r\}$ and let $D=\{v\in V(G):B_v\subseteq S\}$. Then $\{B_v : v\in V(G)\}$ is a partition of $V(G')$.  Since $|S\cap B_v|=r+1$ when $v\in D$ and $|S\cap B_v|\le r$ otherwise, we have
\begin{equation}\label{P9.eq2}
|S|
=\sum_{v\in V(G)}|S\cap B_v|
\le (r+1)|D|+r(n-|D|)
=rn+|D|.
\end{equation}

We claim that $D$ is a $\varepsilon$-mutual visible set of $G$. Let $u,v\in D$ be distinct. Since $u_r,v_r\in S$ and $S$ is $(2r+\varepsilon)$-mutual visible set, there exists a shortest $(u_r,v_r)$-path containing at most $2r+\varepsilon$ internal vertices from $S$. Every shortest $(u_r,v_r)$-path has the form
$(u_r u_{r-1}\ldots u_1 P v_1\ldots v_{r-1}v_r)$,
where $P$ is a shortest $(u,v)$-path in $G$. The vertices $u_{r-1},\ldots,u_1,u,v,v_1,\ldots,v_{r-1}$ all belong to $S$ and already contribute exactly $2r$ internal vertices from $S$, so the internal vertices of $P$ contain at most $\varepsilon$ vertices of $S$, and hence at most $\varepsilon$ vertices of $D$. Thus $u$ and $v$ are $(D,\varepsilon)$-visible in $G$, proving that $D$ is a $\varepsilon$-mutual visible set. Consequently, $|D|\le\mu_\varepsilon(G)$.

Combining this with \eqref{P9.eq2}, we obtain $|S|\le rn+\mu_\varepsilon(G)$. Since $S$ is maximum, $\mu_{2r+\varepsilon}(G')\le rn+\mu_\varepsilon(G)$. Together with the lower bound, this yields
$\mu_{2r+\varepsilon}(G')=rn+\mu_\varepsilon(G)$.
\end{proof}
\begin{defn}$k$-MUTUAL VISIBILITY problem:
\leavevmode\par
\textbf{Instance:} A graph $G$, a nonnegative integer $k$, and a positive integer $Q\le |V(G)|$.

\textbf{Question:} Is $\mu_k(G)\ge Q ?$
\end{defn}
The following theorem is proved by a polynomial-time reduction from the classical NP-complete problem \textsc{3-SAT} \cite{Karp1972}.
\begin{theorem}\label{P9.th9}
The decision problem \textsc{$1$-Mutual Visibility} is NP-hard.
\end{theorem}
\begin{proof}
We prove NP-hardness by giving a polynomial-time reduction from
\textsc{3-SAT} to \textsc{$1$-Mutual Visibility}. Let $\Phi=C_1\wedge C_2\wedge\cdots\wedge C_m$ be a $3$-CNF formula over the variables $x_1,x_2,\ldots,x_p, p \geq 2$. We construct, in polynomial time, a graph $G_\Phi=(V, E)$ and an integer $K=5p+m+4$ such that $\Phi$ is satisfiable if and only if $\mu_1(G_\Phi)\ge K$.

For each variable $x_i$, introduce a variable gadget $T_i$ with vertex set
$V(T_i)={a_i,\bar a_i,u_i,\bar u_i,b_i,\bar b_i}$ and edge set
${a_i\bar a_i,a_i u_i,\bar a_i u_i,u_i\bar u_i,\bar u_i b_i,\bar u_i\bar b_i,b_i\bar b_i}$.
We refer to $a_i$ and $\bar a_i$ as the left vertices of $T_i$, to $b_i$ and $\bar b_i$ as the right vertices of $T_i$, and to $u_i$ and $\bar u_i$ as the literal vertices of $T_i$.

For a literal $\ell$, define
\[
v(\ell)=
\begin{cases}
u_i, & \text{if } \ell=x_i,\\
\bar u_i, & \text{if } \ell=\bar x_i.
\end{cases}
\]

For each clause $C_j=(\ell_{j1}\vee \ell_{j2}\vee \ell_{j3})$, introduce a new vertex $v_j$, and add the edges
$v_jv(\ell_{jt}), \ t=1,2,3$.
Moreover, introduce two vertices $w_1,w_2$ and add the edges $v_jw_1, v_jw_2$ for every $j\in\{1,2,\ldots,m\}$. Also add the edges $
a_iw_1, \bar a_iw_1, b_iw_2, \bar b_iw_2$ for every $i\in\{1,2,\ldots,p\}$.

Finally, introduce five vertices
$y,y_1,c,z,z_1$. For every $i\in\{1,2,\ldots,p\}$, add the edges $
a_iy, \bar a_iy,$
$ u_ic, \bar u_ic, b_iz, \bar b_iz$. Also add edges
$y_1y, yc, cz, zz_1$. Note that the maximum
$1$-mutual visibility sets of $T_i$ have cardinality $5$, and the only such
sets are $V(T_i)\setminus \{u_i\}
\ \text{and}\
V(T_i)\setminus \{\bar u_i\}$.

\begin{figure}
    \centering
\begin{tikzpicture}[
    scale=1, rotate=0,
    every node/.style={font=\small},
    selected/.style={circle,draw=black,fill=red!85,minimum size=2mm,inner sep=0pt},
    omitted/.style={circle,draw=black,fill=white,minimum size=2mm,inner sep=0pt},
    base/.style={draw=black,line width=0.8pt},
    aux/.style={draw=gray!95,line width=0.7pt},
    darkaux/.style={draw=black!85,line width=0.8pt},
    literal/.style={draw=gray!55,line width=0.7pt}
]
\newcommand{\varGadget}[5]{
\begin{scope}[shift={(#2,0)}]
    \node[selected,label=left:{$a_{#3}$}]       (#1-a)    at (-1.35, 0.75) {};
    \node[selected,label=left:{$\bar a_{#3}$}]  (#1-abar) at (-1.35,-0.75) {};

    \node[#4,label=above:{$u_{#3}$}]            (#1-u)    at (-0.35,0) {};
    \node[#5,label=above:{$\bar u_{#3}$}]       (#1-ub)   at ( 0.55,0) {};

    \node[selected,label=right:{$b_{#3}$}]      (#1-b)    at ( 1.55, 0.75) {};
    \node[selected,label=right:{$\bar b_{#3}$}] (#1-bbar) at ( 1.55,-0.75) {};

    \draw[base] (#1-a)--(#1-abar);
    \draw[base] (#1-a)--(#1-u);
    \draw[base] (#1-abar)--(#1-u);
    \draw[base] (#1-u)--(#1-ub);
    \draw[base] (#1-ub)--(#1-b);
    \draw[base] (#1-ub)--(#1-bbar);
    \draw[base] (#1-b)--(#1-bbar);
\end{scope}
}

\varGadget{Xone}{-6.2}{1}{omitted}{selected}
\varGadget{Xtwo}{-2.1}{2}{selected}{omitted}
\varGadget{Xp}{5.6}{p}{selected}{omitted}

\node at (1.75,0) {\Large $\cdots$};

\node[selected,label=above:{$v_1$}] (vone) at (-5.1,2.35) {};
\node[omitted,label=above:{$v_q$}]  (vq)   at ( 0.0,2.35) {};
\node[selected,label=above:{$v_m$}] (vm)   at ( 5.1,2.35) {};

\node at (-2.6,2.35) {\Large $\cdots$};
\node at ( 2.6,2.35) {\Large $\cdots$};

\node[selected,label=above:{$w_1$}] (w1) at (-2.7,4.0) {};
\node[selected,label=above:{$w_2$}] (w2) at ( 2.7,4.0) {};

\node[omitted,label=below:{$y$}]    (y)  at (-2.7,-3.0) {};
\node[selected,label=below:{$y_1$}] (y1) at (-4.1,-3.0) {};
\node[selected,label=below:{$c$}]   (c)  at ( 0.0,-3.0) {};
\node[omitted,label=below:{$z$}]    (z)  at ( 2.7,-3.0) {};
\node[selected,label=below:{$z_1$}] (z1) at ( 4.1,-3.0) {};

\draw[base] (y1)--(y)--(c)--(z)--(z1);

\foreach \v in {vone,vq,vm}{
    \draw[cyan] (\v)--(w1);
    \draw[aux] (\v)--(w2);
}

\foreach \X in {Xone,Xtwo,Xp}{
    \draw[cyan] (w1)--(\X-a);
    \draw[cyan] (w1)--(\X-abar);

    \draw[aux] (w2)--(\X-b);
    \draw[aux] (w2)--(\X-bbar);
}

\foreach \X in {Xone,Xtwo,Xp}{
    \draw[blue] (\X-a)--(y);
    \draw[blue] (\X-abar)--(y);

    \draw[magenta] (\X-u)--(c);
    \draw[magenta] (\X-ub)--(c);

    \draw[aux] (\X-b)--(z);
    \draw[aux] (\X-bbar)--(z);
}

\draw[orange] (vone)--(Xone-u);
\draw[orange] (vone)--(Xtwo-ub);
\draw[orange] (vone)--(Xp-u);

\draw[literal] (vq)--(Xone-ub);
\draw[literal] (vq)--(Xtwo-u);
\draw[literal] (vq)--(Xp-ub);

\draw[literal] (vm)--(Xone-u);
\draw[literal] (vm)--(Xtwo-u);
\draw[literal] (vm)--(Xp-ub);
\end{tikzpicture}
    \caption{Red vertices indicate a typical selected maximum 1-mutual visible set; white vertices indicate omitted vertices.(For interpretation of the references to colour in this ﬁgure legend,
the reader is referred to the web version of this article.)}
    \label{fig:placeholder}
\end{figure}

For each $i\in{1,\ldots,p}$, the variable gadget $T_i$ and the subgraph $H=G_\Phi[{y_1,y,c,z,z_1}]$ are convex in $G_\Phi$. Moreover, each singleton subgraph $L_j=G_\Phi[{v_j}]$, where $j=1,\ldots,m$, as well as $W_1=G_\Phi[{w_1}]$ and $W_2=G_\Phi[{w_2}]$, is convex. Since the vertex sets of these convex subgraphs partition $V(G_\Phi)$, Lemma~\ref{P9.prop5}, together with Proposition~\ref{P9.prop1}, yields
\[ \mu_1(G_\Phi) \le \sum_{i=1}^{p}\mu_1(T_i) + \mu_1(H) + \sum_{j=1}^{m}\mu_1(L_j) + \mu_1(W_1) + \mu_1(W_2)=5p+3+m+2=5p+m+5. \]

\noindent\textbf{Claim.}
Let
$
R=\{v_1,\ldots,v_m,w_1,w_2,y,y_1,c,z,z_1\}.
$
If $S$ is a $1$-mutual visibility set of $G_\Phi$ with
$|S|\ge 5p+m+4$, then
$|S\cap V(T_i)|=5$ for every $i\in\{1,\ldots,p\}$,
$|S\cap R|=m+4$, and hence $|S|=5p+m+4$.
Moreover, $S\cap R$ contains $w_1$, $w_2$, $c$, exactly $m-1$ clause vertices, one vertex from each of $\{y,y_1\}$ and $\{z,z_1\}$.

\smallskip
\noindent
\textbf{proof of the claim.} 
Each variable gadget $T_i$ is convex in $G_{\Phi}$, and hence $|S\cap V(T_i)|\le\mu_1(T_i)=5$ for every $i$. Therefore,
$\sum_{i=1}^{p}|S\cap V(T_i)|\le 5p$. Since $|S|\ge 5p+m+4$, it follows that $|S\cap R|\ge m+4$.
Let $d=5p-\sum_{i=1}^{p}|S\cap V(T_i)|$ and $q=|R\setminus S|$. Since $|R|=m+7$, we have $|S\cap R|=m+7-q$. Consequently,
$|S|=\sum_{i=1}^{p}|S\cap V(T_i)|+|S\cap R|=(5p-d)+(m+7-q)=5p+m+7-(d+q)$.
Since $5p+m+4\le |S|\le\mu_1(G_\Phi)\le 5p+m+5$, it follows that $2\le d+q\le 3$. It remains to show that $q\le2$ is impossible.

\smallskip
\noindent\textbf{Case 1.} $q=0$.
Then every vertex of $R$ belongs to $S$. Since the unique shortest
$(y_1,z_1)$-path is $(y_1,y,c,z,z_1)$, it contains three internal vertices
of $S$. Hence $y_1$ and $z_1$ are not $(S,1)$-visible, a contradiction.

\smallskip
\noindent\textbf{Case 2.} $q=1$.
Exactly one vertex of $R$ is omitted from $S$. If $y_1,z_1\in S$, then at least two of $y,c,z$ also belong to $S$. Hence the unique shortest $(y_1,z_1)$-path contains at least two internal vertices of $S$, contradicting the $(S,1)$-visibility of $y_1$ and $z_1$.

Now suppose that the omitted vertex is $y_1$. Then $w_1,w_2,c,z,z_1\in S$. Every shortest $(w_1,z_1)$-path has one of the forms
$(w_1,\ell,y,c,z,z_1)$,
$(w_1,v_j,s,c,z,z_1)$, or
$(w_1,v_j,w_2,r,z,z_1)$,
where $\ell\in\{a_i,\bar a_i\}$, $s\in\{u_i,\bar u_i\}$, and $r\in\{b_i,\bar b_i\}$. In each case, the path contains at least two internal vertices of $S$. Hence $w_1$ and $z_1$ are not $(S,1)$-visible, a contradiction.

The case in which $z_1$ is omitted is similar, yielding that $w_2$ and $y_1$ are not $(S,1)$-visible. Therefore, $q=1$ is impossible.

\smallskip
\noindent\textbf{Case 3.} $q=2$.
Exactly two vertices of $R$ are omitted from $S$. We distinguish three
subcases according to the number of omitted clause vertices.

\smallskip
\noindent\textbf{Subcase 3.1.} Two clause vertices are omitted. Then no vertex of
$\Omega=\{w_1,w_2,y,y_1,c,z,z_1\}$ is omitted. By the argument of
Case~1, $y_1$ and $z_1$ are not $(S,1)$-visible, a contradiction.

\smallskip
\noindent\textbf{Subcase 3.2.} Exactly one clause vertex is omitted. Then exactly one vertex of $\Omega$ is omitted. If $y_1,z_1\in S$, then at least two of $y,c,z$ also belong to $S$. Hence
$y_1$ and $z_1$ are not $(S,1)$-visible, a contradiction. Now suppose that $y_1\notin S$. Then
$w_1,w_2,c,z,z_1\in S$. By the argument of Case~2,
$w_1$ and $z_1$ are not $(S,1)$-visible, a contradiction.
The case in which $z_1$ is omitted is similar.

\smallskip
\noindent\textbf{Subcase 3.3.} No clause vertex is omitted. Then every clause vertex belongs to $S$, and exactly two vertices of
$\Omega$ are omitted.

Suppose that $y_1,z_1\in S$. To avoid the obstruction on the unique
shortest path $(y_1,y,c,z,z_1)$, at least two of $y,c,z$ must be omitted.
Hence $w_1,w_2,y_1,z_1\in S$. Since $q=2$ and $d+q\le3$, we have
$d\le1$. Therefore, at least one variable gadget $T_i$ contributes five
vertices to $S$, so exactly one of $u_i,\bar u_i$ is omitted. If $u_i\notin S$, then $\bar u_i,b_i,\bar b_i\in S$. Choose
$r\in\{b_i,\bar b_i\}$. Every shortest $(w_1,r)$-path is of the form
$(w_1,v_k,w_2,r)$ or $(w_1,v_j,\bar u_i,r)$, and in either case contains
at least two internal vertices of $S$. Hence $w_1$ and $r$ are not
$(S,1)$-visible. If $\bar u_i\notin S$, a similar argument shows that
$w_2$ is not $(S,1)$-visible with each of $a_i$ and $\bar a_i$.

Now suppose that exactly one of $y_1,z_1$ is omitted. Assume that
$y_1\notin S$. If $w_1\in S$, then, by the argument of Case~2, $w_1$ and $z_1$ are not
$(S,1)$-visible, a contradiction.
If $w_1\notin S$, then the omitted vertices of $R$ are precisely
$y_1$ and $w_1$. Hence
$v_1,\ldots,v_m,w_2,y,c,z,z_1\in S$. For any clause vertex $v_j$, every
shortest $(v_j,z_1)$-path is of the form
$(v_j,s,c,z,z_1)$ or $(v_j,w_2,r,z,z_1)$, where
$s\in\{u_i,\bar u_i\}$ and $r\in\{b_i,\bar b_i\}$. In the first path, the internal vertices $c$ and $z$ belong to $S$; in the second, the internal vertices $w_2$ and $z$ belong to $S$. Hence $v_j$ and $z_1$ are not $(S,1)$-visible. The case in which $z_1$ is omitted is similar.

Finally, suppose that both $y_1$ and $z_1$ are omitted. Then
$v_1,\ldots,v_m,w_1,w_2,y,c,z\in S$. Every shortest $(w_1,z)$-path is
obtained from one of the shortest $(w_1,z_1)$-paths described in
Case~2 by deleting the edge $zz_1$. Consequently, each such path
contains at least two internal vertices of $S$, and therefore $w_1$ and
$z$ are not $(S,1)$-visible.

Thus every possible configuration with $q=2$ leads to a contradiction.
Hence $q\ge3$. Since $d+q\le3$ and $d\ge0$, we obtain
$q=3$ and $d=0$. Consequently,
$|S\cap R|=m+4$ and
$\sum_{i=1}^{p}|S\cap V(T_i)|=5p$.
Therefore each variable gadget contributes exactly five vertices, and
$|S|=5p+m+4$.

Next, we show that exactly one clause vertex is omitted. Since $q=3$,
exactly three vertices of $R$ are omitted from $S$. Suppose, to the
contrary, that no clause vertex is omitted. Then the three omitted
vertices all belong to $\Omega$. Since $d=0$, every variable gadget
contributes exactly five vertices to $S$. Consequently, for every $i$,
exactly one of $u_i$ and $\bar u_i$ is omitted, while all four vertices
$a_i,\bar a_i,b_i,\bar b_i$ belong to $S$.

If $w_1,w_2\in S$, then, by the argument of Subcase~3.3, either
$w_1$ is not $(S,1)$-visible with any selected right vertex or
$w_2$ is not $(S,1)$-visible with any selected left vertex, according as
$u_i\notin S$ or $\bar u_i\notin S$, respectively. Hence
$w_1,w_2\in S$ is impossible. Thus at least one of $w_1$ and $w_2$ is omitted. Assume that
$w_1\notin S$. If $w_2\notin S$, then exactly one vertex of the path
$(y_1,y,c,z,z_1)$ is omitted. Consequently, the two selected endmost
vertices on this path are not $(S,1)$-visible. Hence
$w_1,w_2\notin S$ is impossible.

Suppose that $w_1\notin S$ and $w_2\in S$. Then exactly two vertices of
$\{y,y_1,c,z,z_1\}$ are omitted.
Table~\ref{P9.tab1} lists, for each possible choice of omitted vertices, a selected pair that is not $(S,1)$-visible. Here
$\ell\in\{a_i,\bar a_i\}$ and
$r\in\{b_k,\bar b_k\}$ are selected left and right vertices,
respectively, chosen from distinct variable gadgets. Similarly, the case $w_2\notin S$ and $w_1\in S$ is impossible.

Therefore, the assumption that all clause vertices belong to $S$ leads to a
contradiction. Hence at least one clause vertex is omitted from $S$.
\begin{table}[ht]  \centering \renewcommand{\arraystretch}{1.35}
\begin{tabular}{p{3.5cm}|c|p{6.8cm}}
\hline
\textbf{Omitted vertices from $\{y,y_1,c,z,z_1\}$}
&
\textbf{Selected pair}
&
\textbf{Reason} \\
\hline
$\{y,c\},\ \{y_1,c\}$
&
$w_2,z_1$
&
Every shortest path contains $z$ and either $r$ or $v_j$. \\
\hline
$\{y,z\},\{c,z\},$ $ \{c,z_1\},\ \{z,z_1\}$
&
$w_2,y_1$
&
Every shortest path contains two selected vertices among $r,z,c,y,v_j$. \\
\hline
$\{y,y_1\}$
&
$\ell,z_1$
&
Every shortest path contains $c,z$ as selected internal vertices. \\
\hline
$\{y,z_1\},\ \{y_1,z_1\}$
&
$\ell,r$
&
Every shortest $\ell,r$-path uses either $c,z$ or $v_j,w_2$ as selected internal vertices. \\
\hline
$\{y_1,z\}$
&
$w_2,y$
&
Every shortest path contains $c$ and either $r$ or $v_j$. \\
\hline
\end{tabular}\caption{Selected pairs that are not $(S,1)$-visible for the possible omissions from $\{y,y_1,c,z,z_1\}$.} \label{P9.tab1} \end{table}

If at least two of the omitted vertices are clause vertices, then at most one vertex of $\Omega$ is omitted from $S$. Hence at least four vertices of the path $(y_1,y,c,z,z_1)$ belong to $S$, contradicting the $(S,1)$-visibility of its two selected endmost vertices. Therefore, exactly one clause vertex is omitted from $S$.

It remains to determine which vertices of $\Omega$ belong to $S$.
Since $|R|=m+7$ and $|S\cap R|=m+4$, exactly three vertices of $R$ are omitted
from $S$. As exactly one of them is the clause vertex, say $v_q$, exactly two
vertices are omitted from
$\Omega=\{w_1,w_2,y,y_1,c,z,z_1\}$.

We first show that $w_1,w_2\in S$. Suppose, to the contrary, that
$w_1\notin S$. If $w_2\notin S$, then
$y,y_1,c,z,z_1\in S$. By the argument of Case~1,
$y_1$ and $z_1$ are not $(S,1)$-visible, a contradiction.
Hence $w_2\in S$. If the second omitted vertex belongs to $\{y,c,z\}$, then
$y_1,z_1\in S$. Again, by the argument of Case~2,
$y_1$ and $z_1$ are not $(S,1)$-visible, a contradiction. If the second omitted vertex is $y_1$, then
$z,z_1,w_2\in S$. Since every variable gadget contributes five vertices,
both $b_i$ and $\bar b_i$ belong to $S$ for every $i$. Hence the shortest
paths $(w_2,b_i,z,z_1)$ and $(w_2,\bar b_i,z,z_1)$ each contain two selected
internal vertices. Therefore $w_2$ and $z_1$ are not $(S,1)$-visible, a
contradiction. If the second omitted vertex is $z_1$, then
$y_1,y,c,z\in S$, and the unique shortest $(y_1,z)$-path contains two
selected internal vertices. Hence $y_1$ and $z$ are not $(S,1)$-visible, a
contradiction. Therefore $w_1\in S$. By a similar argument, $w_2\in S$.

Next, we show that $c\in S$. Suppose, to the contrary, that $c\notin S$.
Since $w_1,w_2\in S$, the other omitted vertex of $\Omega$ belongs to
$\{y,y_1,z,z_1\}$. If the second omitted vertex is $y$ or $y_1$, then
$z,z_1\in S$. Every shortest $(w_2,z_1)$-path is of the form
$(w_2,b_i,z,z_1)$ or $(w_2,\bar b_i,z,z_1)$, and hence contains two
selected internal vertices. Therefore $w_2$ and $z_1$ are not
$(S,1)$-visible, a contradiction. If the second omitted vertex is $z$ or $z_1$, a similar argument shows that
$w_1$ and $y_1$ are not $(S,1)$-visible, a contradiction. Hence $c\in S$.

Therefore, the two omitted vertices of $\Omega$ must belong to
$\{y,y_1,z,z_1\}$. If both $y$ and $y_1$ are omitted, then $z,z_1,w_2\in S$, and
$w_2$ and $z_1$ are not $(S,1)$-visible. Similarly, if both $z$ and
$z_1$ are omitted, then $y,y_1,w_1\in S$, and $w_1$ and $y_1$ are not
$(S,1)$-visible. Hence exactly one vertex is omitted from each of
$\{y,y_1\}$ and $\{z,z_1\}$.

Equivalently, $S\cap R$ contains $w_1,w_2,c$, exactly $m-1$ clause
vertices, one vertex from each of $\{y,y_1\}$ and $\{z,z_1\}$. This
completes the proof of the claim.

Next, we establish the equivalence.

Suppose that $\Phi$ is satisfiable, and let $\alpha$ be a satisfying
truth assignment. For each variable $x_i$, define
\[
S_i=
\begin{cases}
V(T_i)\setminus\{u_i\}, & \text{if }\alpha(x_i)=\mathrm{true},\\[1mm]
V(T_i)\setminus\{\bar u_i\}, & \text{if }\alpha(x_i)=\mathrm{false}.
\end{cases}
\]

Choose one clause $C_q$. Since $\alpha$ satisfies $C_q$, it contains a true
literal. Define
\[
S=
\left(\bigcup_{i=1}^{p}S_i\right)
\cup
\{v_j: j\ne q\}
\cup
\{w_1,w_2,y_1,c,z_1\}.
\]
Then $|S|=5p+(m-1)+5=5p+m+4=K$.

We claim that $S$ is a $1$-mutual visibility set. Since each variable gadget
$T_i$ and $H$ are convex, $S_i$ is a $1$-mutual visibility set of $T_i$ for
every $i$, and $S\cap V(H)=\{y_1,c,z_1\}$ is a $1$-mutual visibility set of
$H$, it suffices to verify the $(S,1)$-visibility of pairs of selected
vertices that do not both belong to the same convex subgraph.

For two selected clause vertices $v_j$ and $v_k$, the path
$(v_j,w_1,v_k)$ has only one internal vertex from $S$. Now let
$v_j\in S$ be a selected clause vertex. Since $C_j$ is satisfied, choose
a true literal $\ell_j$ of $C_j$. Then $v(\ell_j)\notin S$. Hence the
paths $(v_j,v(\ell_j),c,y,y_1)$ and
$(v_j,v(\ell_j),c,z,z_1)$ witness the $(S,1)$-visibility of
$v_j$ with $y_1$ and $z_1$, respectively. Also, the path
$(v_j,v(\ell_j),c)$ shows that $v_j$ and $c$ are $(S,1)$-visible.

Now let $v_q$ be the omitted clause vertex. Since $C_q$ is satisfied,
choose a true literal $\ell_q$ of $C_q$. Then $v(\ell_q)\notin S$. The
path $(w_1,v_q,w_2)$ witnesses the $(S,1)$-visibility of $w_1$ and
$w_2$. Moreover, the paths
$(w_1,v_q,v(\ell_q),c)$ and
$(w_2,v_q,v(\ell_q),c)$ witness the $(S,1)$-visibility of the pairs
$(w_1,c)$ and $(w_2,c)$, respectively. Finally, the paths
$(w_1,v_q,v(\ell_q),c,z,z_1)$ and
$(w_2,v_q,v(\ell_q),c,y,y_1)$ witnesses the
$(S,1)$-visibility of the pairs $(w_1,z_1)$ and $(w_2,y_1)$,
respectively.

For every $i$, let
$A_i=\{a_i,\bar a_i\}$ and $B_i=\{b_i,\bar b_i\}$.
If $x\in A_i$ and $x'\in A_k$ with $i\neq k$, then
$(x,y,x')$ is a shortest $(x,x')$-path. Hence every pair of selected
left vertices is $(S,1)$-visible. Similarly, if
$r\in B_i$ and $r'\in B_k$ with $i\neq k$, then
$(r,z,r')$ is a shortest $(r,r')$-path. Hence every pair of selected
right vertices is $(S,1)$-visible. Now let $x\in A_i$ and $r\in B_k$ with $i\neq k$. Then
$(x,y,c,z,r)$ is a shortest $(x,r)$-path, whose only selected internal
vertex is $c$. Hence $x$ and $r$ are $(S,1)$-visible.

Let $s_i$ denote the selected literal vertex of $T_i$. If
$x\in\{a_k,\bar a_k\}$ with $k\neq i$, then
$(s_i,c,y,x)$ is a shortest $(s_i,x)$-path. Likewise, if
$r\in\{b_k,\bar b_k\}$ with $k\neq i$, then
$(s_i,c,z,r)$ is a shortest $(s_i,r)$-path. Each of these paths has $c$
as its only selected internal vertex.

All remaining pairs are either adjacent or are witnessed by one of the following shortest paths:
\[
(a_i,y,y_1),\quad
(\bar a_i,y,y_1),\quad
(b_i,z,z_1),\quad
(\bar b_i,z,z_1),
\]
\[
(a_i,y,c),\quad
(\bar a_i,y,c),\quad
(b_i,z,c),\quad
(\bar b_i,z,c),
\]
\[
(a_i,y,c,z,z_1),\quad
(\bar a_i,y,c,z,z_1),\quad
(b_i,z,c,y,y_1),\quad
(\bar b_i,z,c,y,y_1),
\]
\[
(u_i,c,y,y_1),\quad
(\bar u_i,c,y,y_1),\quad
(u_i,c,z,z_1),\quad
(\bar u_i,c,z,z_1),
\]
\[
(a_i,w_1,v_j),\quad
(\bar a_i,w_1,v_j),\quad
(b_i,w_2,v_j),\quad
(\bar b_i,w_2,v_j),
\]
\[
(w_1,a_i,y,y_1),\quad
(w_2,b_i,z,z_1).
\]
Each displayed path contains at most one selected internal vertex. Therefore, $S$ is a $1$-mutual visibility set of cardinality $K$, and hence $\mu_1(G_\Phi)\ge K$.

Conversely, suppose that $\mu_1(G_\Phi)\ge K$, and let $S$ be a $1$-mutual visibility set with $|S|\ge K$. By the claim, $|S\cap V(T_i)|=5$ for every $i$. Hence, for each $i$, the omitted vertex of $T_i$ is exactly one of $u_i$ and $\bar u_i$. Define a truth assignment $\alpha$ by setting $\alpha(x_i)=\mathrm{true}$ if $u_i\notin S$, and $\alpha(x_i)=\mathrm{false}$ if $\bar u_i\notin S$. Again by the claim, exactly one clause vertex, say $v_q$, is omitted from $S$, and $w_1,w_2,c\in S$. Moreover, $S$ contains one vertex from each of ${y,y_1}$ and ${z,z_1}$.

Let $C_j$ be a clause with $v_j\in S$. Since $v_j$ and the selected
vertex of $\{z,z_1\}$ are $(S,1)$-visible, there exists a shortest path
between them containing at most one internal vertex of $S$. Every such
shortest path passes through a literal vertex of $C_j$ and then through
$c$. Since $c\in S$, it follows that, for some
$t\in\{1,2,3\}$, $v(\ell_{jt})\notin S$. By the definition of $\alpha$,
the literal $\ell_{jt}$ is true. Hence $C_j$ is satisfied.

It remains to consider the omitted clause $C_q$. Since $w_1$ and the
selected vertex of $\{z,z_1\}$ belong to $S$, they are
$(S,1)$-visible. Hence there exists a shortest path between them
containing at most one internal vertex of $S$. Such a path must pass
through a clause vertex, a literal vertex of the corresponding clause,
and $c$. Since $c\in S$, the clause vertex must be $v_q$, and therefore,
for some $t\in\{1,2,3\}$, we have $v(\ell_{qt})\notin S$. By the
definition of $\alpha$, the literal $\ell_{qt}$ is true. Hence $C_q$ is
also satisfied.

Therefore every clause of $\Phi$ is satisfied by $\alpha$, and hence
$\Phi$ is satisfiable. Since the construction is polynomial in the size of $\Phi$, this is a
polynomial-time reduction from \textsc{3-SAT} to
\textsc{$1$-Mutual Visibility}. Therefore,
\textsc{$1$-Mutual Visibility} is NP-hard.
\end{proof}
\begin{theorem}\label{P9.th10}
For every fixed integer $k\ge 0$, the decision problem
\textsc{$k$-Mutual Visibility} is NP-complete.
\end{theorem}

\begin{proof}
Membership in NP follows from Theorem~\ref{P9.th5}, since Algorithm
\textsc{kMV} verifies whether a given subset is a $k$-mutual visibility
set in polynomial time.

The case $k=0$ was proved in \cite{Stefano}. Hence assume that $k\ge1$. If $k=1$, NP-hardness follows from Theorem~\ref{P9.th9}. Therefore,
\textsc{$1$-Mutual Visibility} is NP-complete.

Now let $k\ge2$. If $k$ is even, write $k=2r$, where $r\ge1$. Given an instance
$(G,Q)$ of \textsc{Mutual Visibility}, construct $G'$ from $G$ as in
Theorem~\ref{P9.th7}. Since $
\mu_k(G')=\mu_{2r}(G')=rn+\mu(G),$
where $n=|V(G)|$, setting $Q'=rn+Q$ gives $\mu(G)\ge Q
\iff
\mu_k(G')\ge Q'.$ Hence \textsc{Mutual Visibility} polynomially reduces to
\textsc{$k$-Mutual Visibility}.

If $k$ is odd, write $k=2r+1$, where $r\ge1$. By
Theorem~\ref{P9.th7}, $
\mu_k(G')=\mu_{2r+1}(G')=rn+\mu_1(G).$ Thus, setting $Q'=rn+Q$ yields $\mu_1(G)\ge Q
\iff
\mu_k(G')\ge Q'.$
Hence \textsc{$1$-Mutual Visibility} polynomially reduces to
\textsc{$k$-Mutual Visibility}. Since
\textsc{$1$-Mutual Visibility} is NP-hard by
Theorem~\ref{P9.th9}, it follows that
\textsc{$k$-Mutual Visibility} is NP-hard.

In each case, the construction adds exactly $rn$ vertices and $rn$
edges, and is therefore computable in polynomial time. Consequently,
\textsc{$k$-Mutual Visibility} is NP-hard. Since it belongs to NP, it is
NP-complete.
\end{proof}
\section{$k$-mutual visibility covering number}
\begin{defn}
Let $G$ be a graph. 
A $k$-mutual visibility cover of $G$ is a partition 
$\mathcal{P}$ of $V(G)$ such that each part of 
$\mathcal{P}$ is a  $k$-mutual visible set in $G$. 
The $k$-mutual visibility covering number of $G$, 
denoted by $\tau_{\mathrm{k}}(G)$, is the minimum 
cardinality of a $k$-mutual visibility cover of $G$.
\end{defn}
\begin{remark}\label{P11.rem1}
The mutual $0$-visibility covering number of $G$
coincides with the mutual-visibility chromatic number introduced
in~\cite{MVCN1}.
\end{remark}
\begin{prop}\label{P11.prop1}
Let $G$ be a graph and let $0\le k\le \ell$.
Then $\tau_\ell(G)\le \tau_k(G)$.
\end{prop}

\begin{proof}
Let $\mathcal{P}$ be a $k$-mutual visibility cover of $G$.
Since every  $k$-mutual visible set is also $\ell$-mutual visible, every part of $\mathcal{P}$
is $\ell$-mutual visible. Hence $\mathcal{P}$ is also a mutual $\ell$-visibility cover,
and so $\tau_\ell(G)\le |\mathcal{P}|=\tau_k(G)$.
\end{proof}
\begin{lemma}\label{P11.lem1}
Let $G$ be a graph of order $n$ and let $k\ge 0$.
Then $\tau_k(G)\le \left\lceil \frac{n}{k+2}\right\rceil$.
\end{lemma}

\begin{proof}
Partition $V(G)$ into 
$t=\left\lceil n/(k+2)\right\rceil$ parts, each of size at most $k+2$.
Let $P$ be a part and let $u,v\in P$ be distinct.
Let $Q$ be a shortest $(u,v)$-path in $G$.
Any internal vertex of $Q$ that lies in $P$ belongs to $P\setminus\{u,v\}$,
so there are at most $|P|-2\le k$ such vertices.
Hence $u$ and $v$ are $(P,k)$-visible, and therefore $P$ is a  $k$-mutual visible set.
It follows that this partition is a $k$-mutual visibility cover of $G$ of size $t$,
and the result follows.
\end{proof}
\begin{lemma}\label{P11.lem2}
Let $G$ be a graph of order $n$ and let $k\ge 0$. Then
\[
\left\lceil \frac{n}{\mu_k(G)}\right\rceil \le \tau_k(G)
\le 1+\left\lceil \frac{n-\mu_k(G)}{k+2}\right\rceil .
\]
\end{lemma}
\begin{proof}
Let $\mathcal{P}$ be a $k$-mutual visibility cover of $G$. 
Since each part of $\mathcal{P}$ is a  $k$-mutual visible set, its size is at most $\mu_k(G)$. 
Hence $n=\sum_{P\in\mathcal{P}} |P| \le |\mathcal{P}|\,\mu_k(G)$, which gives the lower bound.

For the upper bound, let $M\subseteq V(G)$ be a  $k$-mutual visible set of maximum cardinality, so $|M|=\mu_k(G)$. 
Partition $V(G)\setminus M$ into parts of size at most $k+2$. 
Each such part is  $k$-mutual visible, and together with $M$ this yields a $k$-mutual visibility cover of size 
$1+\left\lceil (n-\mu_k(G))/(k+2)\right\rceil$.
\end{proof}
\begin{prop}\label{P11.prop4}
Let $G$ be a graph with diameter $d$ and let $k\ge 0$. Then $\tau_k(G)=1$ if and only if $k\ge d-1$.
\end{prop}
\begin{proof}
Since $G$ is connected, $\tau_k(G)=1$ if and only if $V(G)$
is a  $k$-mutual visible set, that is, if and only if
$\mu_k(G)=|V(G)|$.
By Theorem~\ref{P9.th3}, this holds precisely when
$k\ge d-1$.
\end{proof}
\begin{theorem}\label{P11.prop5}
Let $n\ge 1$ and let $k\ge 0$. Then $\tau_k(P_n)=\left\lceil \frac{n}{k+2}\right\rceil$.
\end{theorem}
\begin{proof}
Since each part of a $k$-mutual visibility cover is a  $k$-mutual visible set,
its cardinality is at most $\mu_k(P_n)$.
Hence $\tau_k(P_n)\ge \left\lceil n/\mu_k(P_n)\right\rceil$.
By Proposition~\ref{P9.prop1}, $\mu_k(P_n)=\min\{n,k+2\}$,
and therefore $\left\lceil n/\mu_k(P_n)\right\rceil
=\left\lceil n/\min\{n,k+2\}\right\rceil
=\left\lceil n/(k+2)\right\rceil$.
Thus $\tau_k(P_n)\ge \left\lceil n/(k+2)\right\rceil$.
On the other hand, Lemma~\ref{P11.lem1} yields
$\tau_k(P_n)\le \left\lceil n/(k+2)\right\rceil$.
Hence $\tau_k(P_n)=\left\lceil n/(k+2)\right\rceil$.
\end{proof}
\begin{remark}
The bound in Lemma~\ref{P11.lem1} is sharp, as equality holds for paths by Proposition~\ref{P11.prop5}.
\end{remark}
\begin{theorem}\label{P11.prop2}
Let $n\ge 3$ and let $k\ge 0$. Then
$\tau_k(C_n)=\left\lceil \frac{n}{\min\{n,2k+3\}}\right\rceil$.
In particular, if $2k+3<n$, then
$\tau_k(C_n)=\left\lceil n/(2k+3)\right\rceil$.
\end{theorem}
\begin{proof}
If $2k+3\ge n$, then $\mu_k(C_n)=n$ by Proposition~\ref{P9.prop2}, so $V(C_n)$ itself is a  $k$-mutual visible set and hence $\tau_k(C_n)=1=\left\lceil n/\min\{n,2k+3\}\right\rceil$.

Assume now that $2k+3<n$ and set $t=\left\lceil n/(2k+3)\right\rceil$. Write $C_n=(v_0,v_1,\dots,v_{n-1})$ and partition $V(C_n)$ into $t$ parts according to congruence classes modulo $t$, that is, $v_i$ and $v_j$ lie in the same part if and only if $i\equiv j\pmod t$. Let $P$ be one such part and let $u,v\in P$ be distinct. If $Q$ is a shortest $(u,v)$-path in $C_n$, then $|E(Q)|\le\lfloor n/2\rfloor$. Along $Q$, vertices of $P$ appear at indices differing by multiples of $t$, so $Q$ contains at most $\lfloor |E(Q)|/t\rfloor+1$ vertices of $P$. Hence the number of internal vertices of $Q$ that lie in $P$ is at most $\lfloor |E(Q)|/t\rfloor-1\le \lfloor (n/2)/t\rfloor-1$. Since $t=\lceil n/(2k+3)\rceil$, we have $(n/2)/t\le (2k+3)/2<k+2$, and therefore $\lfloor (n/2)/t\rfloor\le k+1$, which implies $\lfloor (n/2)/t\rfloor-1\le k$. Thus $u$ and $v$ are $(P,k)$-visible, so each part is a  $k$-mutual visible set. Consequently, $\tau_k(C_n)\le t$.

On the other hand, by Lemma~\ref{P11.lem2}, $\tau_k(C_n)\ge \lceil n/\mu_k(C_n)\rceil$, and by Proposition~\ref{P9.prop2}, $\mu_k(C_n)=2k+3$ under the present assumption. Hence $\tau_k(C_n)\ge \lceil n/(2k+3)\rceil=t$. Therefore $\tau_k(C_n)=t$, completing the proof.
\end{proof}
\section{Conclusion}
We introduced $k$-mutual visibility as a natural generalization of mutual visibility in graphs by permitting a bounded number of internal vertices of the selected set to lie on shortest paths. This parameter provides a unified framework that interpolates between classical mutual visibility and unrestricted shortest-path connectivity. We established fundamental properties of the $k$-mutual visibility number, including general bounds in terms of graph parameters, and determined its exact value for several important graph classes. We also investigated the behavior of $k$-mutual visibility in convex graphs and block graphs, where a characterization in terms of $k$-admissible sets of the associated block-cut tree was obtained. From an algorithmic perspective, we presented a polynomial-time algorithm, \textsc{kMV}, for recognizing $k$-mutual visibility sets. We further introduced the \textsc{$k$-Mutual Visibility} decision problem, proved its NP-completeness, and proposed the notion of the $k$-mutual visibility covering number, establishing several general bounds and exact values for fundamental graph classes.

The results presented here provide a foundation for the study of parameterized visibility in graphs. Several interesting problems remain open, including determining
$\mu_k(G)$ for other graph classes, obtaining exact values under
additional graph operations, and establishing sharper bounds in terms of
structural graph parameters.
\section*{Declarations}
\subsection*{Funding sources}
This research did not receive any specific grant from funding agencies in the public, commercial, or not-for-profit sectors.
\subsection*{Conflict of interest}
 The authors declare that they have no conflict of interest.

 \subsection*{Data Availability}No data were used for the research described in the article.

\bibliographystyle{plainurl}
\bibliography{cas-refs}

\end{document}